\documentclass[11pt]{article}
\usepackage{amssymb}
\usepackage{amsmath}
\usepackage{amsthm}
\usepackage{graphicx}
\usepackage[margin=0.5in]{geometry}
\usepackage{bm}
\usepackage{amsthm}
\usepackage{amscd}
\usepackage{amsbsy}
\usepackage{amsmath}
\usepackage{amsfonts}
\usepackage{amssymb}
\usepackage{calligra}
\usepackage{xcolor}
\usepackage{float}
\usepackage[T1]{fontenc}
\usepackage{multirow}
\usepackage{mathrsfs}
\usepackage{mathtools}
\usepackage{graphicx}
\usepackage{epstopdf}
\usepackage{subfigure}
\usepackage{booktabs}
\usepackage{listings}
\usepackage{longtable}
\usepackage{latexsym}
\usepackage{arydshln}
\usepackage{enumerate}
\usepackage{epsfig}
\usepackage{cite}
\usepackage{verbatim}
\usepackage{overpic}
\usepackage{abstract}
\allowdisplaybreaks[4]

\def\hs{\hspace{0.3cm}}

\newtheorem {theorem*}{Theorem}
\newtheorem {theorem} {Theorem}

\newtheorem{proposition}{Proposition}
\newtheorem{lemma}{Lemma}
\newtheorem{corollary}{Corollary}
\newtheorem{remark}{Remark}

\numberwithin{equation}{section}
\numberwithin{lemma}{section}
\numberwithin{theorem}{section}
\numberwithin{proposition}{section}
\numberwithin{corollary}{section}

\usepackage[colorlinks,
            CJKbookmarks=true,
            linkcolor=blue,
            anchorcolor=red,
             urlcolor=blue,
            citecolor=blue
            ]{hyperref}

\begin{document}
\arraycolsep=1pt

\title{\Large\bf  Toeplitz operators on Bergman spaces with exponential weights
\footnotetext{\hspace{-0.35cm}
\endgraf{\it E-mail: yiyuanzhang@e.gzhu.edu.cn (Yiyuan Zhang)}
\endgraf \hspace{1.1cm} {\it wxf@gzhu.edu.cn (Xiaofeng Wang)}
\endgraf \hspace{1.1cm} {\it huzj@hutc.zj.cn (Zhangjian Hu)}
\endgraf This project is supported by the National Natural Science
Foundation of China (Grant Nos. 11971125, 11771139).
}}
\author{Yiyuan Zhang$^{1}$, Xiaofeng Wang$^{1, }$\thanks{Corresponding author}, Zhangjian Hu$^{2}$\\
\small \em 1. School of Mathematics and Information Science and Key Laboratory of Mathematics and Interdisciplinary\\ \small \em Sciences of the Guangdong Higher Education Institute, Guangzhou University, Guangzhou, 510006, China\\
\small \em 2. Department of Mathematics, Huzhou Teachers College, Huzhou 313000, Zhejiang, China}
\date{ }
\maketitle

\vspace{-0.8cm}

\begin{center}
\begin{minipage}{16cm}\small
{\noindent{\bf Abstract} \quad In this paper, we focus on the weighted Bergman spaces $A_{\varphi}^{p}$ in $\mathbb{D}$ with $\varphi\in\mathcal{W}_{0}$. We first give characterizations of those finite positive Borel measures $\mu$ in $\mathbb{D}$ such that the embedding $A_{\varphi}^{p}\subset L_{\mu}^{q}$ is bounded or compact for $0<p,q<\infty$. Then we describe bounded  or compact Toeplitz operators $T_{\mu}$ from one Bergman space $A_{\varphi}^{p}$ to another $A_{\varphi}^{q}$ for all possible $0<p,q<\infty$. Finally, we characterize Schatten class Toeplitz operators on $A_{\varphi}^{2}$.
\endgraf{\bf Mathematics Subject Classification (2010).}\quad Primary 47B35; Secondary 32A36; 47B10.
\endgraf{\bf Keywords.}\quad Toeplitz operators, Bergman spaces, Carleson measures, Schatten class.}
\end{minipage}
\end{center}

\section{Introduction}


Let $\mathbb{C}$ denote the complex plane, $\mathbb{R}$ be the real line and $H(\mathbb{D})$ be the space of analytic functions in the unit disc $\mathbb{D}=\{z\in\mathbb{C}:|z|<1\}$.
For $0<p<\infty$ and a subharmonic function $\varphi$ on $\mathbb{D}$, the weighted Bergman space $A_{\varphi}^{p}$ consists of $f\in H(\mathbb{D})$ such that
$$\left\|f\right\|_{A_{\varphi}^{p}}=\left\{\int_{\mathbb{D}}
\left|f(z)e^{-\varphi(z)}\right|^{p}dA(z)\right\}^{\frac{1}{p}}<\infty,$$
where
$dA(z)=\frac{1}{\pi}dxdy$
denotes the normalized Lebesgue area measure on $\mathbb{D}$.

Let $C_{0}$ be the space of all continuous functions $\rho$ on $\mathbb{D}$ satisfying $\lim_{|z|\rightarrow1}\rho(z)=0$. The class $\mathcal{L}$ is set to be
$$\mathcal{L}=\left\{\rho:\mathbb{D}\rightarrow\mathbb{R}:\|\rho\|_{\mathcal{L}}=\sup_{z,w\in\mathbb{D},z\neq w}\frac{|\rho(z)-\rho(w)|}{|z-w|}<\infty,\ \rho\in C_{0}\right\}.$$

We define the class $\mathcal{L}_{0}$ to be the family of those $\rho\in\mathcal{L}$ with the property that for each $\varepsilon>0$, there exists a compact subset $E\subset \mathbb{D}$ such that
$|\rho(z)-\rho(w)|\leq\varepsilon|z-w|$,
whenever $z, w\in \mathbb{D}\backslash E$.

Now we give the definition of weights to be considered in this paper, that is
$$\mathcal{W}_{0}=\left\{\varphi\in C^{2}:\Delta\varphi>0, \ \text{and}\ \exists \ \rho\in \mathcal{L}_{0}\ \text{such} \ \text{that} \frac{1}{\sqrt{\Delta\varphi}}\asymp\rho\right\},$$
where $\Delta$ denotes the standard Laplace operator and the notation $a\asymp b$  indicates that there exists some positive constant $C$ such that $C^{-1}b\leq a \leq Cb$. The weight $\mathcal{W}_{0}$ covers a large class of weights, including those rapidly radial decreasing weights that are decreasing faster than any standard weight $(1-|z|^{2})^{\alpha}( \alpha>-1)$. Two classes of weights related closely to $\mathcal{W}_{0}$ are worth discussing here. One of the weights was introduced by Oleinik \cite{Oleinik1} in 1978, denoted by $\mathcal{OP}$, which has later been studied in \cite{Constantin1,Constantin2,Galanopoulos1, Arroussi1 ,Pau1,Pau2}. Another class of weight was introduced by Borichev, Dhuez and Kellay \cite{Borichev1} in 2007, denoted by $\mathcal{BDK}$, which has later been considered in \cite{Asserda1,Arroussi0,Arroussi1,Lin1,Lin2}. As stated in \cite{Hu1}, the weight $\mathcal{W}_{0}$ covers $\mathcal{BDK}$, but there is no inclusion relation between the weight $\mathcal{W}_{0}$ and $\mathcal{OP}$. It deserves to be mentioned that the weight $\mathcal{W}_{0}$ contains non-radial weights. Since spaces induced by non-radial weights vary greatly from that of radial weights, plenty of essential problems have not yet been solved. For example, polynomials may not be dense in the  Bergman spaces $A_{\varphi}^{p}$ if $\varphi$ is non-radial, hence new tools must be developed in order to better study on $A_{\varphi}^{p}$.

In what follows, we will focus on the Bergman spaces $A_{\varphi}^{p}$ induced by exponential type weights, that is those weights $w=e^{-p\varphi}$, where $\varphi\in\mathcal{W}_{0}$ and $0<p<\infty$. It turns out that Bergman spaces with exponential weights are similar in spirit to Fock spaces, and some classical techniques used in Fock spaces can achieve certain success in this paper.

It follows from \cite[Lemma 3.3]{Hu1} that there exists a reproducing kernel $K_{z}(\cdot)=K(\cdot,z)$ for $A_{\varphi}^{2}$ and $z\in\mathbb{D}$. By the Riesz representation theorem, the kernel has the reproducing property that
\begin{align*}
f(z)=\int_{\mathbb{D}}f(w)K(z,w)e^{-2\varphi(w)}dA(w)
\end{align*}
for all $f\in A_{\varphi}^{2}$.

For $z\in\mathbb{D}$, let $k_{z}=K_{z}/\left\|K_{z}\right\|_{A_{\varphi}^{2}}$ be the normalized reproducing kernels of $A_{\varphi}^{2}$. Let $D(z, r)$ denote the Euclidean disc centered at $z$ with radius $r > 0$, and for simplicity of notations, we write $D^{r}(z)$ for the disc $D(z,r\rho(z))$.
Given a positive Borel measure $\mu$ on $\mathbb{D}$ and $r>0$, the Berezin transform $\widetilde{\mu}$ and the averaging function $\widehat{\mu}_{r}$ with respect to $\mu$ are defined respectively to be
\begin{align*}
\widetilde{\mu}(z) = \int_{\mathbb{D}}\left|k_{z}(w)\right|^{2} e^{-2\varphi(w)} d \mu(w),\ z\in \mathbb{D},
\end{align*}
and
\begin{align*}
\widehat{\mu}_{r}(z)=\frac{\mu(D^{r}(z))}{\rho(z)^{2}},
\ \ z\in\mathbb{D}.
\end{align*}

For $0<p<\infty$ and a positive Borel measure $\mu$ on $\mathbb{D}$, let
$$L_{\mu}^{p}=\left\{f\ \text{Lebesgue}\ \text{measurable}:
\int_{\mathbb{D}}\left|f(z)e^{-\varphi(z)}\right|^{p}d\mu(z)<\infty\right\}.$$

The first question to be considered in this article is: When is the embedding from the Bergman space $A_{\varphi}^{p}$ to $L_{\mu}^{q}$  bounded or compact? This is also called the Carleson measure problem. Recall that for a finite positive Borel measure $\mu$ on $\mathbb{D}$ and $0<p,q<\infty$, $\mu$ is said to be a $q$-Carleson measure for $A_{\varphi}^{p}$ if the identity operator $Id:A_{\varphi}^{p}\rightarrow L_{\mu}^{q}$ is bounded. Correspondingly, $\mu$ is said to be a vanishing $q$-Carleson measure for $A_{\varphi}^{p}$ if the identity operator $Id:A_{\varphi}^{p}\rightarrow L_{\mu}^{q}$ is compact.
After  Carleson's  pioneering works \cite{Carleson1,Carleson2}, there have been a large number of researches on this problem and some characterizations  have been obtained on several spaces of analytic functions. In \cite{Duren0}, Duren obtained the necessary and sufficient condition for the finite measure $\mu$ on $\mathbb{D}$ such that the embedding from Hardy space $H^{p}$ to $L_{\mu}^{q}$ is bounded for $0<p\leq q<\infty$. In 2006, Girela and  Pel\'{a}ez \cite{Girela1} got an equivalent characterization of the positive Borel measure $\mu$ on $\mathbb{D}$ for which the Dirichlet type space $\mathcal{D}_{\alpha}^{p}\subseteq L_{\mu}^{q}$, where $0<p<q<\infty$ and $\alpha >-1$.
In the year 2010, the corresponding question was considered by  Pau and Pel\'{a}ez in \cite{Pau1}, they have completely described those positive Borel measures $\mu$ on $\mathbb{D}$ such that the the weighted Bergman space $A^{p}_{\omega}\subset L^{q}_{\mu}$, where $\omega$ belongs to $\mathcal{BDK}$ and $0<p,q<\infty$.
In 2014, Hu and Lv \cite{Hu3} characterize those positive Borel measures
 on $\mathbb{C}^{n}$ such that the embedding from the Fock space $F^{p}(\varphi)$ to  $L_{\mu}^{q}$ is bounded or compact for $0<p,q<\infty$.
Recently, Wang etc. \cite{Wang2} gave several equivalent characterizations on this problem for Fock-Type spaces.
In this paper, we consider this question on  a more general
setting of Bergman spaces, and we have obtained several equivalent conditions for this question in terms of the Berezin transform and the averaging function in section \ref{section3}.

Given a finite positive Borel measure $\mu$ on $\mathbb{D}$, Toeplitz operator $T_{\mu}$ associated with symbol $\mu$ is defined by
$$T_{\mu}f(z)=\int_{\mathbb{D}}f(w)K(z, w)e^{-2\varphi(w)}d\mu(w), \ \ z\in\mathbb{D}.$$

There have been plenty of researches about Toeplitz operators acting on different spaces of analytic functions  and the theory is especially well understood in the case of standard Fock spaces \cite{Zhu1} or Bergman spaces \cite{Zhu2}.
Luecking \cite{Luecking2} was probably the pioneer to study Toeplitz operators $T_{\mu}$ with measures as symbols, and the study of Toeplitz operators  on  Bergman spaces with exponential weights was initiated by Lin and Rochberg \cite{Lin2} in 1996.

The second question we are interested in is:
Under what conditions for the finite positive Borel measure $\mu$  can the induced Toeplitz operators $T_{\mu}$ be bounded or compact from one Bergman space $A_{\varphi}^{p}$ to another $A_{\varphi}^{q}$? This question has been extensively studied on Fock spaces by authors \cite{Wang1,Wang2,Hu2,Hu3,Hu4}, but less research on  Bergman spaces. In \cite[Chapter 7]{Zhu1}, bounded or compact Toeplitz operators $T_{\mu}$ have been characterized in terms of the Berezin transform and Carleson measures  on the standard Bergman space $A_{\alpha}^{2}$ $(\alpha>-1)$.
In 2016,  Pel\'{a}ez etc. \cite{Pelaez1} provided characterizations of the bounded and compact Toeplitz operators $T_{\mu}$ between different Bergman spaces $A^{p}_{\omega}$, $A^{q}_{\omega}$ induced by regular weights $\omega$ in terms of Carleson measures and the Berezin transform for $1<p,q <\infty$.
The main purpose of this paper is to characterize those positive Borel measures for which the induced Toeplitz operators $T_{\mu}$ are bounded or compact from $A_{\varphi}^{p}$ to $A_{\varphi}^{q}$ for all possible $0<p,q<\infty$. 
We accomplish this work by using the tools of the Berezin transform and the averaging function in section \ref{section4}.

The final question we concern about is: When does the Toeplitz operator $T_{\mu}$ belong to the Schatten $p$-class  $\mathcal{S}_{p}(A_{\varphi}^{2})$?
This question was first considered by Luecking \cite{Luecking2} on the standard Bergman spaces $A^{2}_{\alpha}$, and one
can also  find in Zhu's book \cite{Zhu1}. In 2015, Arroussi etc. \cite{Arroussi1} considered the same problem and gave descriptions of membership in the Schatten $p$-class $\mathcal{S}_{p}(A_{\omega}^{2})$, where the weight $\omega$ belongs to $\mathcal{OP}$.
Our characterizations on this question are shown in section \ref{section5}.

Throughout this manuscript, we let $C$ denote a positive constant whose exact value may change from one occurence to another but do not depend on the variables being considered.
We use the notation $a\lesssim b$ to indicate that there is a constant $C > 0$ such that $a\leq Cb$. Similarly for the notation $\gtrsim$.  For each $1\leq p<\infty$, we let $p^{\prime}$ for its conjugate exponent, that is, $1/p+1/p^{\prime}=1$.
Besides, in the proof of our main theorems, we will always assume that the positive  Borel measure $\mu$ satisfies $\int_{\mathbb{D}}d\mu(z)<\infty$.
\section{Preliminaries}
In this section, we are going to present some basic conclusions that will be used in the following chapters. We start with a classical covering lemma, which is stated by Hu etc. in \cite{Hu1}, and a sequence $\{w_{k}\}$ satisfying $(a)-(c)$ is called a $(\rho, r)$-lattice.
\begin{lemma}\label{latticelemma}
If $\rho\in\mathcal{L}$ is positive, then there exist positive constants $\alpha$ and $s$, depending only on
$\|\rho\|_{\mathcal{L}}$, such that for $0<r\leq\alpha$ there is a sequence $\{w_{k}\}\subset\mathbb{D}$ satisfying\\
(a) $\mathbb{D}=\cup_{k}D^{r}(w_{k})$;\\
(b) $D^{sr}(w_{k})\cap D^{sr}(w_{j})=\emptyset$ for $k\neq j$;\\
(c) $\{D^{2\alpha}(w_{k})\}_{k}$ is a covering of $\mathbb{D}$ of finite multiplicity $N$.
\end{lemma}

From now on, we will assume that $\alpha$ is chosen such that the conclusions of  Lemma \ref{latticelemma} and \cite [Lemma 3.1]{Hu1} are valid for it.

The following Lemma from \cite [Lemma 3.3]{Hu1} plays an essential role in proving our main results and can be regarded as one type of generalized sub-mean property of $|fe^{-\varphi}|^{p}$ which describes the boundedness of the point evaluation functionals on $A_{\varphi}^{p}$.

\begin{lemma}\label{submeanlemma}
Suppose $\varphi\in\mathcal{W}_{0}$ with $\frac{1}{\sqrt{\Delta\varphi}}\asymp\rho\in\mathcal{L}_{0}$ and $0 <p<\infty$. Then there exist positive constants $\alpha$ and $C$ such that the following holds
\begin{align}
\left|f(z)e^{-\varphi(z)}\right|^{p}\leq C\frac{1}{\rho(z)^{2}}\int_{D^{r}(z)}\left|f(w)e^{-\varphi(w)}\right|^{p}dA(w)
\end{align}
for $r\in(0, \alpha]$ and $f\in H(\mathbb{D})$.
\end{lemma}
In order to give the pointwise estimate of the Bergman kernels, it is necessary to consider the distance $d_{\rho}$ induced by $\rho\in\mathcal{L}_{0}$, which is given by
$$d_{\rho}(z,w)=\inf_{\gamma}\int_{0}^{1}|
\gamma^{\prime}(t)|\frac{dt}{\rho(\gamma(t))},\ \ z,\ w\in\mathbb{D},$$
where the infimum is taken over all piecewise $C^{1}$ curves $\gamma:[0, 1]\rightarrow\mathbb{D}$ with $\gamma(0)=z$ and $\gamma(1)=w$.


The next lemma gives the upper bound estimate of the reproducing kernel $K(z,w)$ for all $z,w\in\mathbb{D}$ and the lower bound estimate of $K(z,w)$ near the diagonal, one can see \cite [Theorem 3.2]{Hu1} for the detailed proof.
\begin{lemma}\label{kernelestimate1}
Let $\varphi\in\mathcal{W}_{0}$ with $\frac{1}{\sqrt{\Delta\varphi}}\asymp\rho\in\mathcal{L}_{0}$. There exist positive constants $C_{1}, C_{2}, \sigma$ such that
\begin{align}\label{ke11}
\left|K(z,w)\right|\leq C_{1}\frac{e^{\varphi(z)+\varphi(w)}}{\rho(z)\rho(w)}e^{-\sigma d_{\rho}(z,w)},\ \ z,\ w\in\mathbb{D},
\end{align}
and
\begin{align}\label{ke12}
|K(z,w)|\geq C_{2}\frac{e^{\varphi(z)+\varphi(w)}}{\rho(z)\rho(w)}, \ \ w\in D^{\alpha}(z).
\end{align}
\end{lemma}

The next results will be frequently used in this paper, which gives the asymptotic estimates for the $A_{\varphi}^{p}$-norm of the (normalized) reproducing kernel, see \cite [Corollary 3.2]{Hu1}.
\begin{lemma}\label{kernelestimate2}
Let $\varphi\in\mathcal{W}_{0}$ with $\frac{1}{\sqrt{\Delta\varphi}}\asymp\rho\in\mathcal{L}_{0}$ and $0<p\leq\infty$. Then for $z\in\mathbb{D}$, we have
\begin{align}\label{ke21}
\|K_{z}\|_{A_{\varphi}^{p}}\asymp e^{\varphi(z)}\rho(z)^{\frac{2}{p}-2},
\end{align}
and
\begin{align}\label{ke22}
\|k_{z}\|_{A_{\varphi}^{p}}\asymp \rho(z)^{\frac{2}{p}-1}.
\end{align}
\end{lemma}
By using the duality theory and the estimate for the $A_{\varphi}^{p}$-norm of the reproducing kernel $K_{z}$, we can obtain the so-called partial atomic decomposition of $A_{\varphi}^{p}$ with $0<p<\infty$. For $z\in \mathbb{D}$, we shall use the notation $k_{p,z}$ to denote the normalized reproducing kernel for $A_{\varphi}^{p}$, which is given by $k_{p,z}=K_{z}/\|K_{z}\|_{A_{\varphi}^{p}}$.

\begin{proposition}\label{atomicdecomposition}
Let $0<p<\infty$ and $\varphi\in\mathcal{W}_{0}$ with $\frac{1}{\sqrt{\Delta\varphi}}\asymp\rho\in\mathcal{L}_{0}$. If $\{w_{k}\}$ is a $(\rho, r)$-lattice, then the function
\begin{align*}
F(z)=\sum_{k}c_{k}k_{p,w_{k}}(z)
\end{align*}
belongs to $A_{\varphi}^{p}$ for any sequence $\{c_{k}\}\in l^{p}$.
Moreover, we have
\begin{align}\label{atomicdecomposition1}
\|F\|_{A_{\varphi}^{p}}\lesssim\|\{c_{k}\}\|_{l^{p}}.
\end{align}
\end{proposition}
\begin{proof}
If $0<p\leq1$, by using the following fact
\begin{align}\label{fact0}
\left(\sum_{k=1}^{\infty}a_{k}\right)^{\theta}\leq\sum_{k=1}^{\infty}a_{k}^{\theta},\ a_{k}>0,\ 0<\theta\leq1,
\end{align}
it is not difficult to see
\begin{align}\label{atomicdecomposition2}
\|F\|_{A_{\varphi}^{p}}^{p}=\int_{\mathbb{D}}\left|\sum_{k}c_{k}k_{p,w_{k}}(z)
e^{-\varphi(z)}\right|^{p}dA(z)
&\leq\sum_{k}\frac{|c_{k}|^{p}}{\|K_{w_{k}}\|_{A_{\varphi}^{p}}^{p}}
\int_{\mathbb{D}}\left|K_{w_{k}}(z)
e^{-\varphi(z)}\right|^{p}dA(z)=\sum_{k}|c_{k}|^{p}.
\end{align}

If $p>1$, use estimate (\ref{ke21}) and apply H\"{o}lder's inequality to get
\begin{align}\label{Festimate}
\|F\|_{A_{\varphi}^{p}}^{p}\nonumber&\lesssim\int_{\mathbb{D}}\left(\sum_{k}|c_{k}|
e^{-\varphi(w_{k})}\rho(w_{k})^{2-\frac{2}{p}}|K_{w_{k}}(z)|\right)^{p}e^{-p\varphi(z)}dA(z)\\
&\lesssim\int_{\mathbb{D}}\left(\sum_{k}|c_{k}|^{p}e^{-\varphi(w_{k})}|K_{w_{k}}(z)|\right)
M(z)^{p-1}e^{-p\varphi(z)}dA(z),
\end{align}
where
$$M(z)=\sum_{k}\rho(w_{k})^{2}e^{-\varphi(w_{k})}|K_{w_{k}}(z)|.$$

It follows from Lemmas \ref{submeanlemma}, \ref{latticelemma} and  estimate (\ref{ke21}) that
\begin{align}\label{Mestimate}
M(z)\lesssim\sum_{k}\int_{D^{r}(w_{k})}
|K_{z}(\zeta)|e^{-\varphi(\zeta)}dA(\zeta)
\lesssim\int_{\mathbb{D}}|K_{z}(\zeta)|e^{-\varphi(\zeta)}dA(\zeta)\asymp e^{\varphi(z)}.
\end{align}

Therefore, putting estimate (\ref{Mestimate}) into (\ref{Festimate}) and using  estimate (\ref{ke21}) again, we deduce
\begin{align*}
\|F\|_{A_{\varphi}^{p}}^{p}\nonumber&\lesssim\int_{\mathbb{D}}\left(\sum_{k}|c_{k}|^{p}
e^{-\varphi(w_{k})}\left|K_{w_{k}}(z)\right|\right)e^{-\varphi(z)}dA(z)\\
&\lesssim\sum_{k}|c_{k}|^{p}e^{-\varphi(w_{k})}\int_{\mathbb{D}}
\left|K_{w_{k}}(z)\right|e^{-\varphi(z)}dA(z)
\lesssim\sum_{k}|c_{k}|^{p}.
\end{align*}

Since the sequence $\{c_{k}\}\in l^{p}$ is arbitrary, this together with estimate (\ref{atomicdecomposition2}) implies that $F$ belongs to $A_{\varphi}^{p}$ with $\|F\|_{A_{\varphi}^{p}}\lesssim\|\{c_{k}\}\|_{l^{p}}$. This completes the proof of the proposition.
\end{proof}
\begin{proposition}\label{uniformly}
Let $0<p<\infty$ and $\varphi\in\mathcal{W}_{0}$ with $\frac{1}{\sqrt{\Delta\varphi}}\asymp\rho\in\mathcal{L}_{0}$. Then the set $\left\{k_{p,z}:z\in\mathbb{D}\right\}$ is bounded in $A_{\varphi}^{p}$ and the normalized reproducing kernel $k_{p,z}$ tends to zero uniformly on
compact subsets of $\mathbb{D}$ as $|z|\rightarrow1^{-}$.
\end{proposition}
\begin{proof}
It is trivial that the set $\left\{k_{p,z}:z\in\mathbb{D}\right\}$ is bounded in $A_{\varphi}^{p}$ since  $\|k_{p,z}\|_{A_{\varphi}^{p}}=1$.
According to \cite[Theorem 3.3]{Hu1}, for each positive number $M$, there exists a constant $C>0$ such that
\begin{align*}
|K(w, z)| \leq C \frac{e^{\varphi(z)+\varphi(w)}}{\rho(z) \rho(w)}\left(\frac{\min \{\rho(z), \rho(w)\}}{|z-w|}\right)^{M}, \ \ \ z,w\in\mathbb{D},\ z\neq w.
\end{align*}
Hence, for any $|w|\leq r_{0}<1$, when $M$ is taken large enough, the inequality above together with estimate (\ref{ke21})  gives
\begin{align*}
|k_{p,z}(w)|&\leq C\frac{e^{\varphi(w)}}{\rho(w)}\cdot\frac{\rho(z)^{M+1-\frac{2}{p}}}{|z-w|^{M}}
\leq C\frac{e^{\varphi(w)}}{\rho(w)}\cdot\frac{\rho(z)^{M+1-\frac{2}{p}}}
{\left||z|-r_{0}\right|^{M}}\rightarrow0,\ \ \text{as}\ |z|\ \rightarrow1^{-}.
\end{align*}
Therefore, $k_{p,z}\rightarrow0$  uniformly on compact subsets of $\mathbb{D}$ as $|z|\rightarrow1^{-}$.
\end{proof}
\begin{proposition}\label{corollarysubmean}
Let $0<p<\infty$, $\delta\in(0, \alpha]$ and $\mu$ be a positive Borel measure on $\mathbb{D}$. Then we have
\begin{align*}
\int_{\mathbb{D}}\left|f(z)e^{-\varphi(z)}\right|^{p}d\mu(z)
\lesssim\int_{\mathbb{D}}\left|f(z)e^{-\varphi(z)}\right|^{p}
\widehat{\mu}_{\delta}(z)dA(z)
\end{align*}
for any $f\in H(\mathbb{D})$.
\end{proposition}
\begin{proof}
Similar to that of \cite[Lemma 2.4]{Hu4}.
\end{proof}
Given a measurable function $f$, let $\widetilde{f}$ be the Berezin transform of $f$. If we set $d\mu=fdA$, then we write $\widetilde{f}=\widetilde{\mu}$.

\begin{proposition}\label{Berezinfunctioncompress}
Let $1 \leq p \leq\infty$. Then the operator $f\mapsto\widetilde{f}$ is bounded on $L^{p}$.
\end{proposition}
\begin{proof}
For $w\in\mathbb{D}$, according to estimates (\ref{ke21}), (\ref{ke11}) and \cite[Corollary 3.1]{Hu1}, we have
\begin{align}\label{Bfc1}
\int_{\mathbb{D}}\left|k_{z}(w)e^{-\varphi(w)}\right|^{2}dA(z)
\nonumber&
\asymp\int_{\mathbb{D}}\left|K(w,z)\right|^{2}e^{-2\varphi(w)-2\varphi(z)}\rho(z)^{2}dA(z)\\
&\lesssim\frac{1}{\rho(w)^{2}}\int_{\mathbb{D}}e^{-\sigma d_{\rho}(z,w)}dA(z)
\leq C
\end{align}
for some constant $C>0$.

If $p = 1$, applying  Fubini's
theorem and using estimate (\ref{Bfc1}), we obtain
\begin{align*}
\left\|\widetilde{f}\right\|_{L^{1}}
&\leq\int_{\mathbb{D}}\int_{\mathbb{D}}
\left|k_{z}(w)e^{-\varphi(w)}\right|^{2}\left|f(w)\right|dA(w)dA(z)\\
&=\int_{\mathbb{D}}\left|f(w)\right|dA(w)\int_{\mathbb{D}}
\left|k_{z}(w)e^{-\varphi(w)}\right|^{2}dA(z)
\lesssim\left\|f\right\|_{L^{1}}.
\end{align*}

If $p =\infty$, estimate (\ref{ke22}) gives
\begin{align*}
\left|\widetilde{f}\right|\leq\int_{\mathbb{D}}
\left|k_{z}(w)e^{-\varphi(w)}\right|^{2}\left|f(w)\right|dA(w)
\leq \left\|f\right\|_{L^{\infty}}\int_{\mathbb{D}}
\left|k_{z}(w)e^{-\varphi(w)}\right|^{2}dA(w)
\asymp\left\|f\right\|_{L^{\infty}}.
\end{align*}
Hence, we have $\left\|\widetilde{f}\right\|_{L^{\infty}}\lesssim\left\|f\right\|_{L^{\infty}}$.
Therefore, by interpolation we see that the operator $f\mapsto\widetilde{f}$ is bounded on
$L^{p}$ for all $1 \leq p \leq\infty$.
\end{proof}
\begin{proposition}\label{ThreeEquivalents}
Let $0<p<\infty$ and $\mu$ be a positive Borel measure on $\mathbb{D}$. Then the following statements are equivalent:\\
(a) $\widetilde{\mu}\in L^{p}$;\\
(b) $\widehat{\mu}_{\delta}\in L^{p}$ for some (or any) $\delta\in(0, \alpha]$;\\
(c) The sequence $\left\{\widehat{\mu}_{r}(a_{k})\rho(a_{k})^{2/p}\right\}_{k}\in l^{p}$ for some (or any) $(\rho,r)$-lattice $\{a_{k}\}$ with $r\in(0, \alpha]$.\\
Moreover, we have
\begin{align}\label{TE}
\|\widetilde{\mu}\|_{L^{p}}\asymp\|\widehat{\mu}_{\delta}\|_{L^{p}}
\asymp\left\|\left\{\widehat{\mu}_{r}(a_{k})\rho(a_{k})^{2/p}\right\}_{k}\right\|_{l^{p}}.
\end{align}
\end{proposition}
\begin{proof}
Given $0<p<\infty$, $s\in \mathbb{R}$,  $(\rho,r)$-lattice $\{a_{k}\}$ and $(\rho,\delta)$-lattice $\{b_{k}\}$. 
We first assert that
\begin{align}\label{meanfuncnorm}
\left\|\left\{\widehat{\mu}_{r}(a_{k})\rho(a_{k})^{s+2/p}\right\}_{k}\right\|_{l^{p}}
\asymp\left\|\left\{\widehat{\mu}_{\delta}(b_{k})\rho(b_{k})^{s+2/p}\right\}_{k}\right\|_{l^{p}}.
\end{align}
To obtain this, for a fixed $z\in\mathbb{D}$, consider the set
$$R_{z}=\left\{j:D^{\delta}(z)\cap D^{r}(a_{j})\neq\emptyset\right\},$$
and let $card(R_{z})$ be the cardinality of $R_{z}$. Since there exists a constant $t>0$ such that the set $\{D^{rt}(a_{j})\}_{j}$ are pairwise disjoint by  Lemma \ref{latticelemma}, and
there exist a constant $C>0$ such that
\begin{align}\label{tworho}
C^{-1}\rho(z)\leq\rho(a_{j})\leq C\rho(z)
\end{align}
for  $j\in R_{z}$ by \cite[Lemma 3.1]{Hu1}, it follows that
\begin{align*}
D^{rt}(a_{j})\subset D^{\delta+rtC+rC}(z),\ \text{for}\ j\in R_{z}.
\end{align*}
Thus we have $\bigcup_{j\in R_{z}}D^{rt}(a_{j})\subset D^{\delta+rtC+rC}(z)$. Hence there exists some positive integer $M$ independent of $z$ such that $card(R_{z})\leq M$. Let
\begin{align*}
T_{j,k}=\left\{
\begin{array}{cl}
1, &\ if\ D^{\delta}(b_{k})\cap D^{r}(a_{j})\neq\emptyset,\\
0, &\  if\ D^{\delta}(b_{k})\cap D^{r}(a_{j})=\emptyset.
\end{array}\right.
\end{align*}
Then we have
\begin{align}\label{SM}
\sum_{j}T_{j,k}=card(R_{b_{k}})\leq M.
\end{align}

Symmetrically, for a fixed $z\in\mathbb{D}$, we let
\begin{align*}
L_{z}=\left\{k:D^{\delta}(b_{k})\cap D^{r}(z)\neq\emptyset\right\}.
\end{align*}
From the definition of $L_{z}$ and \cite[Lemma 3.1]{Hu1}, it is easy to see that $D^{r}(a_{j})\subset\bigcup_{k\in L_{a_{j}}}D^{\delta}(b_{k})$ and
$\rho(b_{k})\asymp\rho(a_{j})$ for $k\in L_{a_{j}}$.
Then we arrive at
\begin{align}\label{777}
\widehat{\mu}_{r}(a_{j})\rho(a_{j})^{s+2/p}\lesssim\sum_{k\in L_{a_{j}}}
\widehat{\mu}_{\delta}(b_{k})\rho(b_{k})^{s+2/p}.
\end{align}
Therefore, apply estimates (\ref{777}) and (\ref{SM}), we deduce
\begin{align}\label{888}
\sum_{j}\widehat{\mu}_{r}(a_{j})^{p}\rho(a_{j})^{sp+2}
\nonumber&\lesssim\sum_{j}\left(\sum_{k\in L_{a_{j}}}
\widehat{\mu}_{\delta}(b_{k})\rho(b_{k})^{s+2/p}\right)^{p}\\
&\lesssim\sum_{k}\sum_{j}T_{j,k}\widehat{\mu}_{\delta}(b_{k})^{p}\rho(b_{k})^{sp+2}
\leq M\sum_{k}\widehat{\mu}_{\delta}(b_{k})^{p}\rho(b_{k})^{sp+2}.
\end{align}

By symmetry, we get
\begin{align*}
\sum_{k}\widehat{\mu}_{r}(b_{k})^{p}\rho(a_{k})^{sp+2}
\lesssim \sum_{j}\widehat{\mu}_{\delta}(a_{j})^{p}\rho(a_{j})^{sp+2},
\end{align*}
this together with estimate (\ref{888}) gives  assertion (\ref{meanfuncnorm}).

We can now prove the equivalence $(b)\Leftrightarrow(c)$. In fact, given any $(\rho,\delta)$-lattice $\{b_{k}\}$ with $\delta\in(0, \alpha]$, in views of \cite[Lemma 3.1]{Hu1}, there exists some constant $B>0$ such that
\begin{align*}
D^{\delta/B}(b_{k})\subseteq D^{\delta}(z)\subseteq D^{B\delta}(b_{k}),\ \ \text{for}\ z\in D^{\delta}(b_{k}).
\end{align*}
Hence, this together with Lemma \ref{latticelemma}, \cite[Lemma 3.1]{Hu1} and (\ref{meanfuncnorm}) respectively gives
\begin{align*}
\int_{\mathbb{D}}\widehat{\mu}_{\delta}(z)^{p}dA(z)
\nonumber&\asymp\sum_{k}\int_{D^{\delta}(b_{k})}\widehat{\mu}_{\delta}(z)^{p}dA(z)
\lesssim\sum_{k}\int_{D^{\delta}(b_{k})}\widehat{\mu}_{B\delta}(b_{k})^{p}dA(z)\\
&\asymp\sum_{k}\widehat{\mu}_{B\delta}(b_{k})^{p}\rho(b_{k})^{2}
\asymp\sum_{k}\widehat{\mu}_{r}(a_{k})^{p}\rho(a_{k})^{2},
\end{align*}
and
\begin{align*}
\int_{\mathbb{D}}\widehat{\mu}_{\delta}(z)^{p}dA(z)
\nonumber&\asymp\sum_{k}\int_{D^{\delta}(b_{k})}\widehat{\mu}_{\delta}(z)^{p}dA(z)
\gtrsim\sum_{k}\int_{D^{\delta}(b_{k})}\widehat{\mu}_{\delta/B}(b_{k})^{p}dA(z)\\
&\asymp\sum_{k}\widehat{\mu}_{\delta/B}(b_{k})^{p}\rho(b_{k})^{2}
\asymp\sum_{k}\widehat{\mu}_{r}(a_{k})^{p}\rho(a_{k})^{2},
\end{align*}
for any $(\rho,r)$-lattice $\{a_{k}\}$ with $r\in(0, \alpha]$,
which gives $(b)\Leftrightarrow(c)$ with
\begin{align}\label{meanfuncnorm1}
\left\|\widehat{\mu}_{\delta}\right\|_{L^{p}}\asymp
\left\|\left\{\widehat{\mu}_{r}(a_{k})\rho(a_{k})^{2/p}\right\}_{k}\right\|_{l^{p}}.
\end{align}
Moreover, for fixed $\delta, r\in(0, \alpha]$, we have
\begin{align}\label{meanfuncnorm11}
\left\|\widehat{\mu}_{\delta}\right\|_{L^{p}}\asymp
\left\|\widehat{\mu}_{r}\right\|_{L^{p}}.
\end{align}

Next, we are going to prove $(a)\Rightarrow(b)$. For $z\in\mathbb{D}$, by Lemma \ref{kernelestimate1} and \cite[Lemma 3.1]{Hu1}, we see that there exists some $r_{0}\in(0,\alpha]$ such that
\begin{align}\label{K}
\left|K(w,z)\right|\asymp\frac{e^{\varphi(w)+\varphi(z)}}{\rho(z)^{2}},\ w\in D^{r_{0}}(z).
\end{align}
Hence, it follows by estimate (\ref{ke21}) that
\begin{align}\label{meanfuncestimate}
\widehat{\mu}_{r_{0}}(z)\nonumber&=\frac{1}{\rho(z)^{2}}\int_{D^{r_{0}}(z)}d\mu(w)
\asymp e^{-2\varphi(z)}\rho(z)^{2}\int_{D^{r_{0}}(z)}\left|K(w,z)\right|^{2}e^{-2\varphi(w)}d\mu(w)\\
&\asymp\int_{D^{r_{0}}(z)}|k_{z}(w)|^{2}
e^{-2\varphi(w)}d\mu(w)
\lesssim\int_{\mathbb{D}}|k_{z}(w)|^{2}
e^{-2\varphi(w)}d\mu(w)=\widetilde{\mu}(z),
\end{align}
 this together with (\ref{meanfuncnorm11}) yields statement $(b)$ with
\begin{align}\label{te1}
\left\|\widehat{\mu}_{\delta}\right\|_{L^{p}}\lesssim \left\|\widetilde{\mu}\right\|_{L^{p}}
\end{align}
for  $\delta\in(0, \alpha]$.

It remains to show the implication $(b)\Rightarrow(a)$. If $1\leq p<\infty$,
take $f=k_{z}$ and $p=2$ in Proposition \ref{corollarysubmean}, we obtain
\begin{align*}
\widetilde{\mu}(z)\lesssim\widetilde{\widehat{\mu}_{\delta}}(z),\ \ z\in\mathbb{D}.
\end{align*}
Hence, this together with Proposition \ref{Berezinfunctioncompress} gives
\begin{align}\label{te2}
\|\widetilde{\mu}\|_{L^{p}}\lesssim
\|\widetilde{\widehat{\mu}_{\delta}}\|_{L^{p}}
\lesssim\|\widehat{\mu}_{\delta}\|_{L^{p}}.
\end{align}

If $0<p<1$, for any $(\rho,r)$-lattice $\{a_{k}\}$ with $r\in(0,\alpha]$, note that there exists some constant $B>1$ such that
\begin{align}\label{4545}
\bigcup_{w\in D^{r}(a_{j})}D^{r}(w)\subset D^{Br}(a_{j}),\ \ w \in\mathbb{D},
\end{align}
see \cite[Lemma 3.1 (B)]{Hu1} and its proof. Then we can divide the lattice $\{a_{k}\}$ into $J$ subsequence $\{a_{j,k}\}_{k}$ $(j=1,2,\ldots,J)$  such that each sequence $\{a_{j,k}\}_{k}$  is a
$(\rho,Br)$-lattice. It follows from (\ref{meanfuncnorm1}) that
\begin{align}\label{3434}
\sum_{k}\widehat{\mu}_{Br}(a_{k})^{p}\rho(a_{k})^{2}
=\sum_{j=1}^{J}\sum_{k}\widehat{\mu}_{Br}(a_{j,k})^{p}\rho(a_{j,k})^{2}
\lesssim\|\widehat{\mu}_{\delta}\|_{L^{p}}^{p}
\end{align}
for $\delta\in(0,\alpha]$.

By estimate (\ref{ke21}), Proposition \ref{corollarysubmean}, Lemma \ref{latticelemma}, (\ref{4545}), the fact (\ref{fact0}), estimate (\ref{ke11}) and \cite[Lemma 3.1]{Hu1}, we deduce
\begin{align*}
\left|\widetilde{\mu}(z)\right|^{p}
&\asymp e^{-2p\varphi(z)}\rho(z)^{2p}
\left(\int_{\mathbb{D}}|K(w,z)|^{2}e^{-2\varphi(w)}
d\mu(w)\right)^{p}\\
&\lesssim e^{-2p\varphi(z)}\rho(z)^{2p}
\left(\int_{\mathbb{D}}|K(w,z)|^{2}e^{-2\varphi(w)}\widehat{\mu}_{r}(w)
dA(w)\right)^{p}\\
&\lesssim e^{-2p\varphi(z)}\rho(z)^{2p}
\left(\sum_{j}\int_{D^{r}(a_{j})}|K(w,z)|^{2}e^{-2\varphi(w)}\widehat{\mu}_{r}(w)
dA(w)\right)^{p}\\
&\lesssim e^{-2p\varphi(z)}\rho(z)^{2p}
\left(\sum_{j}\widehat{\mu}_{Br}(a_{j})\int_{D^{r}(a_{j})}|K(w,z)|^{2}
e^{-2\varphi(w)}dA(w)\right)^{p}\\
&\leq e^{-2p\varphi(z)}\rho(z)^{2p}\sum_{j}\widehat{\mu}_{Br}(a_{j})^{p}
\left(\int_{D^{r}(a_{j})}|K(w,z)|^{2}e^{-2\varphi(w)}dA(w)\right)^{p}\\
&\lesssim\sum_{j}\widehat{\mu}_{Br}(a_{j})^{p}\rho(a_{j})^{2p}
\sup_{w\in D^{r}(a_{j})}\rho(w)^{-2p}e^{-2\sigma d_{\rho}(z,w)}\\
&\lesssim\sum_{j}\widehat{\mu}_{Br}(a_{j})^{p}
\sup_{w\in D^{r}(a_{j})}e^{-2\sigma d_{\rho}(z,w)}.
\end{align*}
Integrating both sides above and using \cite[Corollary 3.1]{Hu1}, we get
\begin{align*}
\left\|\widetilde{\mu}\right\|_{L^{p}}^{p}&\lesssim \sum_{j}\widehat{\mu}_{Br}(a_{j})^{p}
\sup_{w\in D^{r}(a_{j})}\int_{\mathbb{D}}e^{-2\sigma d_{\rho}(z,w)}dA(z)\\
&\lesssim \sum_{j}\widehat{\mu}_{Br}(a_{j})^{p}
\sup_{w\in D^{r}(a_{j})}\rho(w)^{2}
\lesssim\sum_{j}\widehat{\mu}_{Br}(a_{j})^{p}\rho(a_{j})^{2}.
\end{align*}
Therefore, this together with estimate (\ref{3434}) yields
\begin{align}\label{te3}
\left\|\widetilde{\mu}\right\|_{L^{p}}^{p}\lesssim\sum_{j}
\widehat{\mu}_{Br}(a_{j})^{p}\rho(a_{j})^{2}\lesssim
\|\widehat{\mu}_{\delta}\|_{L^{p}}^{p}.
\end{align}
Thus the implication $(b)\Rightarrow(a)$ follows from estimates (\ref{te2}) and (\ref{te3}) with
\begin{align}\label{te4}
\left\|\widetilde{\mu}\right\|_{L^{p}}\lesssim
\|\widehat{\mu}_{\delta}\|_{L^{p}},\ \ 0<p<\infty.
\end{align}

The quantity equivalence (\ref{TE}) comes from estimates (\ref{meanfuncnorm1}), (\ref{te1}) and (\ref{te4}).
This completes the proof of Proposition \ref{ThreeEquivalents}.
\end{proof}
\section{Carleson Measures}\label{section3}
In order to  characterize the (vanishing) $q$-Carleson measure for $A_{\varphi}^{p}$, it will be convenient for us to have a family of suitable test functions, which is given in the following lemma and it is an easy consequence of Lemmas \ref{kernelestimate1}, \ref{kernelestimate1} and \cite[Lemma 3.1, Theorem 3.3]{Hu1}.
\begin{lemma}\label{Carlesonlemma}
Suppose $0<p<\infty$ and $\varphi\in\mathcal{W}_{0}$ with $\frac{1}{\sqrt{\Delta\varphi}}\asymp\rho\in\mathcal{L}_{0}$. Then for any $w\in \mathbb{D}$ and $r\in(0,\alpha]$, we have
\begin{align}\label{(Carlesonlemma1)}
|k_{w}(z)|e^{-\varphi(z)}\asymp\frac{1}{\rho(w)}, \ \  \text{if} \ z\in D^{r}(w),
\end{align}
and
\begin{align}\label{(Carlesonlemma2)}
|k_{w}(z)|e^{-\varphi(z)}\lesssim\frac{1}{\rho(z)}\left(\frac{\min\{\rho(z),\rho(w)\}}
{|z-w|}\right)^{N}
,\ z\in\mathbb{D},
\end{align}
for each positive constant $N$.
\end{lemma}

\begin{theorem}\label{carleson}
Let $0<p\leq q<\infty$, $\varphi\in\mathcal{W}_{0}$ with $\frac{1}{\sqrt{\Delta\varphi}}\asymp\rho\in\mathcal{L}_{0}$ and $\mu$ be a finite positive Borel measure on $\mathbb{D}$. Then the following statements are equivalent:\\
(a) $\mu$ is a $q$-Carleson measure for $A_{\varphi}^{p}$;\\
(b) $\widetilde{\mu}\rho^{2-2q/p}\in L^{\infty}$;\\
(c) $\widehat{\mu}_{\delta}\rho^{2-2q/p}\in L^{\infty}$ for some (or any) $\delta\in(0, \alpha]$  small enough;\\
(d) $\left\{\widehat{\mu}_{r}(a_{k})\rho(a_{k})^{2-2q/p}\right\}_{k}\in l^{\infty}$ for some (or any) $(\rho, r)$-lattice $\{a_{k}\}$ with $r\in(0, \alpha]$ small enough.\\
Moreover, we have
\begin{align}\label{carleson1}
\|Id\|^{q}_{A_{\varphi}^{p}\rightarrow A_{\varphi}^{q}}\asymp\left\|\widetilde{\mu}\rho^{2-2q/p}\right\|_{L^{\infty}}
\asymp\left\|\widehat{\mu}_{\delta}\rho^{2-2q/p}\right\|_{L^{\infty}}
\asymp\left\|\left\{\widehat{\mu}_{r}(a_{k})\rho(a_{k})^{2-2q/p}\right\}_{k}\right\|_{l^{\infty}}.
\end{align}
\end{theorem}
\begin{proof}
$(c)\Rightarrow(d)$. Assume that $\widehat{\mu}_{\delta}\rho^{2-2q/p}$ belongs to $L^{\infty}$ for some $\delta\in(0, \alpha]$, then its discrete case $\left\{\widehat{\mu}_{\delta}(b_{k})\rho(b_{k})^{2-2q/p}\right\}_{k}$ is in $l^{\infty}$ for some  $(\rho, \delta)$-lattice $\{b_{k}\}$. Hence the implication $(c)\Rightarrow(d)$ is trivial from (\ref{meanfuncnorm}) with
\begin{align}\label{11lL}
\left\|\left\{\widehat{\mu}_{r}(a_{k}) \rho(a_{k})^{2-2q/p}\right\}_{k}\right\|_{l^{\infty}} \lesssim\left\|\widehat{\mu}_{\delta} \rho^{2-2q/p}\right\|_{L^{\infty}}
\end{align}
for some $(\rho, r)$-lattice $\{a_{k}\}$ with $r\in(0, \alpha]$.

$(b)\Rightarrow(c)$. The implication $(b)\Rightarrow(c)$ follows from estimate (\ref{meanfuncestimate}) with $r_{0}\in(0, \alpha]$ satisfying (\ref{K}). Observe that (\ref{meanfuncnorm11}) is also true for $p=\infty$. These give
\begin{align}\label{11LLL}
\left\|\widehat{\mu}_{\delta} \rho^{2-2q/p}\right\|_{L^{\infty}} \asymp\left\|\widehat{\mu}_{r_{0}} \rho^{2-2q/p}\right\|_{L^{\infty}}
\lesssim\left\|\widetilde{\mu} \rho^{2-2q/p}\right\|_{L^{\infty}}
\end{align}
for any $\delta\in(0, \alpha]$

$(d)\Rightarrow(b)$. Suppose $\left\{\widehat{\mu}_{r}(a_{k})\rho(a_{k})^{2-2q/p}\right\}_{k}\in l^{\infty}$
 for some $(\rho, r)$-lattice $\{a_{k}\}$ with $r$ small enough. Since $r$  is taken small enough, by \cite[Lemma 3.1]{Hu1}, we have
\begin{align}\label{vip}
D^{r}(w)\subset D^{2\alpha}(a),\ \ \text{for}\ w\in D^{r}(a).
\end{align}
Let $s=2p/q$, for any $f\in A_{\varphi}^{s}$, by Lemma \ref{submeanlemma}, \cite[Lemma 3.1]{Hu1}  and (\ref{vip}), we obtain
\begin{align}\label{11submeanvary}
\sup_{w\in D^{r}(a)}\left|f(w)e^{-\varphi(w)}\right|^{s}\leq C\frac{1}{\rho(a)^{2}}\int_{D^{2\alpha}(a)}\left|f(\zeta)
e^{-\varphi(\zeta)}\right|^{s}dA(\zeta).
\end{align}

In views of estimate (\ref{ke21}), we have the following fact
\begin{align}\label{k99}
|k_{z}(w)|^{2} \rho(z)^{2-2q/p} \asymp\left|k_{s, z}(w)\right|^{2},\ \ z, w\in\mathbb{D}.
\end{align}
This together with  Lemma \ref{latticelemma} and  (\ref{11submeanvary}) yields
\begin{align*}
\widetilde{\mu}(z) \rho(z)^{2-2q/p}
&=\int_{\mathbb{D}}\left|k_{z}(w)\right|^{2}\rho(z)^{2-2q/p}e^{-2\varphi(w)}d\mu(w)
\asymp\int_{\mathbb{D}}\left|k_{s, z}(w)\right|^{2}e^{-2\varphi(w)}d\mu(w)\\
&\leq\sum_{k=1}^{\infty}\int_{D^{r}(a_{k})}\left|k_{s, z}(w)\right|^{2}e^{-2\varphi(w)}d\mu(w)\\
&\leq\sum_{k=1}^{\infty}\mu(D^{r}(a_{k}))\left(\sup_{w\in D^{r}(a_{k})}\left|k_{s, z}(w)e^{-\varphi(w)}\right|^{s}\right)^{q/p}\\
&\leq C\sum_{k=1}^{\infty}\widehat{\mu}_{r}(a_{k})\rho(a_{k})^{2-2q/p}
\left(\int_{D^{2\alpha}(a_{k})}\left|k_{s, z}(\zeta)e^{-\varphi(\zeta)}\right|^{s}dA(\zeta)\right)^{q/p}.
\end{align*}
Since $q/p\geq1$, using the fact
\begin{align}\label{fact1}
\sum_{k=1}^{\infty}a_{k}^{\theta}\leq\left(\sum_{k=1}^{\infty}a_{k}\right)^{\theta},\ a_{k}>0,\ 1\leq\theta<\infty,
\end{align}
and  Lemma \ref{latticelemma}, we
deduce
\begin{align*}
\widetilde{\mu}(z) \rho(z)^{2-2q/p}
&\leq C\sup_{k}\widehat{\mu}_{r}(a_{k})\rho(a_{k})^{2-2q/p}
\left(\sum_{k=1}^{\infty}\int_{D^{2\alpha}(a_{k})}\left|k_{s, z}(\zeta)e^{-\varphi(\zeta)}\right|^{s}dA(\zeta)\right)^{q/p}\\
&\leq C N^{q/p}\sup_{k}\widehat{\mu}_{r}(a_{k})\rho(a_{k})^{2-2q/p}\left\|k_{s, z}\right\|_{A_{\varphi}^{s}}^{2}\\
&\lesssim \sup_{k}\widehat{\mu}_{r}(a_{k})\rho(a_{k})^{2-2q/p}.
\end{align*}
This gives $(b)$ with
\begin{align}\label{11Ll}
\left\|\widetilde{\mu}\rho^{2-2q/p}\right\|_{L^{\infty}}\lesssim \left\|\left\{\widehat{\mu}_{r}(a_{k})\rho(a_{k})^{2-2q/p}\right\}_{k}\right\|_{l^{\infty}}.
\end{align}

$(a)\Rightarrow(c)$. Assume that $\mu$ is a $q$-Carleson measure for $A_{\varphi}^{p}$, it follows that
\begin{align*}
\int_{\mathbb{D}}\left|f(z)e^{-\varphi(z)}\right|^{q}d\mu(z)\lesssim\|Id\|^{q}_{A_{\varphi}^{p}\rightarrow L_{\mu}^{q}}\|f\|_{A_{\varphi}^{p}}^{q}
\end{align*}
for any $f\in A_{\varphi}^{p}$. In particular, take the normalized reproducing kernel $k_{w}$, where $w\in\mathbb{D}$. For any $\delta\in(0,\alpha]$, we have
\begin{align*}
\int_{D^{\delta}(w)}\left|k_{w}(z)e^{-\varphi(z)}\right|^{q}d\mu(z)
\leq\int_{\mathbb{D}}\left|k_{w}(z)e^{-\varphi(z)}\right|^{q}d\mu(z)
\lesssim\|Id\|^{q}_{A_{\varphi}^{p}\rightarrow L_{\mu}^{q}}\|k_{w}\|_{A_{\varphi}^{p}}^{q}.
\end{align*}
 Combine this with  estimates (\ref{(Carlesonlemma1)}) and (\ref{ke22}), we obtain
\begin{align*}
\frac{1}{\rho(w)^{q}}\int_{D^{\delta}(w)}d\mu(z)
\lesssim\|Id\|^{q}_{A_{\varphi}^{p}\rightarrow L_{\mu}^{q}}\rho(w)^{\frac{2q}{p}-q}.
\end{align*}
Therefore, we arrive at
\begin{align}\label{Pt1}
\left\|\widehat{\mu}_{\delta}\rho^{2-2q/p}\right\|_{L^{\infty}}
\lesssim\|Id\|^{q}_{A_{\varphi}^{p}\rightarrow L_{\mu}^{q}}.
\end{align}
which gives $(a)\Rightarrow(c)$.

$(c)\Rightarrow(a)$. Suppose $\widehat{\mu}_{\delta}\rho^{2-2q/p}\in L^{\infty}$ for some $\delta\in(0, \alpha]$ small enough, we need to show the identity operator $Id$ is bounded from $A_{\varphi}^{p}$ to $L_{\mu}^{q}$. To this end, we fix an $(\rho, \delta)$-lattice $\{w_{k}\}$, then apply Lemma \ref{latticelemma}, (\ref{11submeanvary}) and Fubini's theorem to get
\begin{align}\label{Pt2}
\int_{\mathbb{D}}\left|f(z)e^{-\varphi(z)}\right|^{q}d\mu(z)\nonumber
&\lesssim\sum_{k=1}^{\infty}\int_{D^{\delta}(w_{k})}\left(\frac{1}{\rho(w_{k})^{2}}\int_{D^{2\alpha}(w_{k})}
\left|f(w)e^{-\varphi(w)}\right|^{p}dA(w)\right)^{q/p}d\mu(z)\\\nonumber
&=\sum_{k=1}^{\infty}\left(\int_{D^{2\alpha}(w_{k})}
\left|f(w)e^{-\varphi(w)}\right|^{p}dA(w)\right)^{q/p}
\widehat{\mu}_{\delta}(w_{k})\rho(w_{k})^{2-2q/p}\\
&\leq \left\|\widehat{\mu}_{\delta}\rho^{2-2q/p}\right\|_{L^{\infty}}
\sum_{k=1}^{\infty}\left(\int_{D^{2\alpha}(w_{k})}
\left|f(w)e^{-\varphi(w)}\right|^{p}dA(w)\right)^{q/p}.
\end{align}

Since $q/p\geq1$, using the fact (\ref{fact1})
and the finite multiplicity $N$ of the covering $\left\{D^{2\alpha}(w_{k})\right\}_{k}$, we
obtain
\begin{align*}
\int_{\mathbb{D}}\left|f(z)e^{-\varphi(z)}\right|^{q}d\mu(z)&\lesssim \left\|\widehat{\mu}_{\delta}\rho^{2-2q/p}\right\|_{L^{\infty}}
\left(\sum_{k=1}^{\infty}\int_{D^{2\alpha}(w_{k})}
\left|f(w)e^{-\varphi(w)}\right|^{p}dA(w)\right)^{q/p}\\
&\lesssim N^{q/p}\left\|\widehat{\mu}_{\delta}\rho^{2-2q/p}\right\|_{L^{\infty}}
\|f\|_{A_{\varphi}^{p}}^{q},
\end{align*}
which implies that the identity operator $Id:A_{\varphi}^{p}\rightarrow L_{\mu}^{q}$ is bounded with $\|Id\|_{A_{\varphi}^{p}\rightarrow L_{\mu}^{q}}^{q}\lesssim \left\|\widehat{\mu}_{\delta}
\rho^{2-2q/p}\right\|_{L^{\infty}}$, and this together with estimates (\ref{11lL}), (\ref{11LLL}), (\ref{11Ll}) and (\ref{Pt1}) gives (\ref{carleson1}).
\end{proof}
\begin{theorem}\label{carlesonA}
Let $0<p\leq q<\infty$, $\varphi\in\mathcal{W}_{0}$ with $\frac{1}{\sqrt{\Delta\varphi}}\asymp\rho\in\mathcal{L}_{0}$ and $\mu$ be a finite positive Borel measure on $\mathbb{D}$. Then the following statements are equivalent:\\
(a) $\mu$ is a vanishing $q$-Carleson measure for $A_{\varphi}^{p}$;\\
(b) $\lim_{|z|\rightarrow1^{-}}\widetilde{\mu}(z)\rho(z)^{2-2q/p}=0$;\\
(c) $\lim_{|z|\rightarrow1^{-}}\widehat{\mu}_{\delta}(z)\rho(z)^{2-2q/p}=0$ for some (or any) $\delta\in(0, \alpha]$ small enough;\\
(d) $\lim_{k\rightarrow\infty}\widehat{\mu}_{r}(a_{k})\rho(a_{k})^{2-2q/p}=0$  for some (or any) $(\rho, r)$-lattice $\{a_{k}\}$ with $r\in(0, \alpha]$ small enough.
\end{theorem}
\begin{proof}
The proof of implications  $(b)\Rightarrow(c)$ and $(c)\Rightarrow(d)$ are similar to the same part of Theorem \ref{carleson}.

$(d)\Rightarrow(b)$. Suppose $\widehat{\mu}_{r}(a_{k})\rho(a_{k})^{2-2q/p}\rightarrow0$ as $k\rightarrow\infty$ for some $(\rho, r)$-lattice $\{a_{k}\}$ with $r\in(0, \alpha]$ small enough. Then for any $\varepsilon>0$, there exists some positive integer $K$ such that $\widehat{\mu}_{r}(a_{k})\rho(a_{k})^{2-2q/p}<\varepsilon$ whenever $k > K$. Since $\bigcup_{k=1}^{K}\overline{D^{2\alpha}(a_{k})}$ is a compact subset of $\mathbb{D}$, by Proposition \ref{uniformly} we see that $\left\{k_{s,z}\right\}_{z\in\mathbb{D}}\subseteq A_{\varphi}^{s}$ uniformly converges to 0 on $\bigcup_{k=1}^{K}\overline{D^{2\alpha}(a_{k})}$ as $|z|\rightarrow1^{-}$, where $s=2p/q$. From estimates (\ref{k99}), (\ref{11submeanvary}), (\ref{fact1}) and Lemma \ref{latticelemma}, as $|z|$ is sufficiently close to $1^{-}$, we deduce
\begin{align*}
\widetilde{\mu}(z) \rho(z)^{2-2q/p}
&\asymp\int_{\mathbb{D}}\left|k_{s, z}(w)\right|^{2}e^{-2\varphi(w)}d\mu(w)\\
&\leq\int_{\bigcup_{k=1}^{K}\overline{D^{2\alpha}(a_{k})}}\left|k_{s, z}(w)\right|^{2}e^{-2\varphi(w)}d\mu(w)+\sum_{k=K+1}^{\infty}\mu(D^{r}(a_{k}))\left(\sup_{w\in D^{r}(a_{k})}\left|k_{s, z}(w)e^{-\varphi(w)}\right|^{s}\right)^{q/p}\\
&<\varepsilon+C\sum_{k=K+1}^{\infty}\widehat{\mu}_{r}(a_{k})\rho(a_{k})^{2-2q/p}
\left(\int_{D^{2\alpha}(a_{k})}\left|k_{s, z}(\zeta)e^{-\varphi(\zeta)}\right|^{s}dA(\zeta)\right)^{q/p}\\
&<\varepsilon+C\sup_{k\geq K+1}\widehat{\mu}_{r}(a_{k})\rho(a_{k})^{2-2q/p}
\left(\sum_{k=K+1}^{\infty}\int_{D^{2\alpha}(a_{k})}\left|k_{s, z}(\zeta)e^{-\varphi(\zeta)}\right|^{s}dA(\zeta)\right)^{q/p}\\
&<\varepsilon+CN^{q/p}\|k_{s,z}\|^{2}_{A_{\varphi}^{s}}\varepsilon
\lesssim\varepsilon.
\end{align*}
This gives that $\widetilde{\mu}(z) \rho(z)^{2-2q/p}\rightarrow0$ as $|z|\rightarrow1^{-}$.

$(c)\Rightarrow(a)$. Assume that $\lim_{|z|\rightarrow1^{-}}\widehat{\mu}_{\delta}(z)\rho(z)^{2-2q/p}=0$ for any $\delta\in(0, \alpha]$ small enough. Given any $\varepsilon>0$, there exists a $\delta_{0}\in(0,1)$ such that
\begin{align}\label{Pt3}
\sup_{|z|>\delta_{0}}\widehat{\mu}_{\delta}(z)\rho(z)^{2-2q/p}<\varepsilon.
\end{align}

Consider the  $(\rho,\delta)$-lattice $\{w_{n}\}$. Notice that there exists a $\delta_{0}^{\prime}<1$ with $\delta_{0}^{\prime}\geq \delta_{0}$ such that if some $w_{k}\in\{w_{n}\}$ belongs to the disc $\{z: |z|\leq \delta_{0}\}$, then $D^{\delta}(w_{k})\subseteq\{w: |w|\leq \delta_{0}^{\prime}\}$.
Hence,
\begin{align}\label{Pt4}
\bigcup_{|w_{k}|\leq \delta_{0}}D^{\delta}(w_{k})\subseteq\{w: |w|\leq \delta_{0}^{\prime}\}.
\end{align}

Let $\{f_{n}\}$ be a bounded sequence in $A_{\varphi}^{p}$, then the sequence is uniformly bounded on compact subsets of $\mathbb{D}$ by Lemma \ref{submeanlemma}, hence it forms a normal family by Montel's theorem, see \cite[Chapter 5]{Ahlfors}. Therefore, there exists a subsequence $\{f_{n_{k}}\}$ converging to an analytic function $f$ uniformly on compact subsets of $\mathbb{D}$. Therefore, choose $n_{k}$ big enough, we have
\begin{align}\label{Pt6}
\sup_{|z|\leq \delta_{0}^{\prime}}|f(z)-f_{n_{k}}(z)|<\varepsilon^{\frac{1}{q}}.
\end{align}
And an application of Fatou's lemma implies that $f$ is in $A_{\varphi}^{p}$.

Let $g_{n_{k}}=f-f_{n_{k}}$, then $\{g_{n_{k}}\}$ is a bounded sequence in $A_{\varphi}^{p}$. We have to show the sequence $\{g_{n_{k}}\}$ is compact in $L_{\mu}^{q}$. For this purpose, we divide the integral region into two parts, use (\ref{Pt4}), (\ref{Pt3}), (\ref{Pt6}) and argue as in (\ref{Pt2}) to obtain
\begin{align*}
\int_{\mathbb{D}}\left|g_{n_{k}}(z)e^{-\varphi(z)}\right|^{q}d\mu(z)
&\leq\sum_{|w_{n}|\leq \delta_{0}}\int_{D^{\delta}(w_{n})}
\left|g_{n_{k}}(z)e^{-\varphi(z)}\right|^{q}d\mu(z)+
\sum_{|w_{n}|> \delta_{0}}\int_{D^{\delta}(w_{n})}
\left|g_{n_{k}}(z)e^{-\varphi(z)}\right|^{q}d\mu(z)\\
&\leq\int_{|w|\leq \delta_{0}^{\prime}}
\left|g_{n_{k}}(z)e^{-\varphi(z)}\right|^{q}d\mu(z)+
\sum_{|w_{n}|> \delta_{0}}\int_{D^{\delta}(w_{n})}
\left|g_{n_{k}}(z)e^{-\varphi(z)}\right|^{q}d\mu(z)\\
&\leq\sup_{|z|\leq \delta_{0}^{\prime}}\left|g_{n_{k}}(z)e^{-\varphi(z)}\right|^{q}\mu(\mathbb{D})+
\sum_{|w_{n}|> \delta_{0}}\int_{D^{\delta}(w_{n})}
\left|g_{n_{k}}(z)e^{-\varphi(z)}\right|^{q}d\mu(z)\\
&\leq C\varepsilon+C\|g_{n_{k}}\|_{A_{\varphi}^{p}}^{q}\sup_{|w_{n}|> \delta_{0}}
\widehat{\mu}_{\delta}(w_{n})\rho(w_{n})^{2-2q/p}<C\varepsilon.
\end{align*}
This implies that the identity operator $Id:A_{\varphi}^{p}\rightarrow L_{\mu}^{q}$ is compact, i.e., $\mu$ is a vanishing $q$-Carleson measure for $A_{\varphi}^{p}$.

$(a)\Rightarrow(c)$. Assume that $\mu$ is a vanishing $q$-Carleson measure for $A_{\varphi}^{p}$. For $z\in\mathbb{D}$, let
\begin{align*}
f_{w}(z)=\frac{k_{w}(z)}{\rho(w)^{2/p-1}},\ r_{0}\leq|w|<1,
\end{align*}
where $r_{0}\in(0,1)$. By estimate (\ref{ke22}), we have
\begin{align*}
\sup_{r_{0}\leq|w|<1}\|f_{w}\|_{A_{\varphi}^{p}}\lesssim 1.
\end{align*}
Hence the set
$\{f_{w}:r_{0}\leq|w|<1\}$ is compact in $L_{\mu}^{q}$ by the compactness the identity operator $Id$ from $A_{\varphi}^{p}$ to $L_{\mu}^{q}$.

Therefore
\begin{align}\label{Pt5}
\lim_{r\rightarrow1^{-}}\int_{r<|z|<1}
\left|f_{w}(z)e^{-\varphi(z)}\right|^{q}d\mu(z)=0
\end{align}
uniformly in $w$.

On the other hand, a consequence of estimate (\ref{(Carlesonlemma2)}) shows
\begin{align*}
\left|f_{w}(z)e^{-\varphi(z)}\right|^{p}\lesssim\frac{\rho(w)^{(N+1)p-2}}
{(1-r)^{Np}},\ \ \ |z|\leq r,\  |w|\geq\frac{1+r}{2}.
\end{align*}
for each  constant $N>0$. Since $\rho\in C_{0}$, when $N$ is chosen big enough, we have
\begin{align*}
\lim_{|w|\rightarrow1^{-}}\int_{|z|\leq r}
\left|f_{w}(z)e^{-\varphi(z)}\right|^{q}d\mu(z)=0.
\end{align*}
This together with (\ref{Pt5}) indicates that
\begin{align*}
\lim_{|w|\rightarrow1^{-}}\|f_{w}\|_{L_{\mu}^{q}}=0.
\end{align*}

Therefore, according to estimate (\ref{(Carlesonlemma1)}), we obtain
\begin{align*}
\sup_{|w|>r}\widehat{\mu}_{\delta}(w)\rho(w)^{2-2q/p}
&\asymp\sup_{|w|>r}\frac{1}{\rho(w)^{2q/p-q}}\int_{D^{\delta}(w)}
\left|k_{w}(z)e^{-\varphi(z)}\right|^{q}d\mu(z)\\
&\asymp\sup_{|w|>r}\int_{D^{\delta}(w)}\left|f_{w}(z)e^{-\varphi(z)}\right|^{q}d\mu(z)
\leqslant\sup_{|w|>r}\|f_{w}\|_{L_{\mu}^{q}}^{q}\rightarrow0
\end{align*}
as $r\rightarrow1^{-}$, which gives the statement $(c)$. This completes the proof of Theorem \ref{carlesonA}.
\end{proof}
\begin{remark}
In the  proof of the above two theorems, the condition of $\delta$ (or $r$) small enough is only used in implications $(d)\Rightarrow(b)$ and $(c)\Rightarrow(a)$. In the case of $0<p\leq q<\infty$, it turns out from Theorems \ref{carleson} and \ref{carlesonA} that the notion of (vanishing) $q$-Carleson measure for $A_{\varphi}^{p}$ is independent of the exact values of $p$ and $q$, but depends only on their ratios.
\end{remark}
\begin{theorem}\label{carlesonB}
Let $0<q<p<\infty$, $\varphi\in\mathcal{W}_{0}$ with $\frac{1}{\sqrt{\Delta\varphi}}\asymp\rho\in\mathcal{L}_{0}$ and $\mu$ be a finite positive Borel measure on $\mathbb{D}$. Then the following statements are equivalent:\\
(a) $\mu$ is a vanishing $q$-Carleson measure for $A_{\varphi}^{p}$;\\
(b) $\mu$ is a $q$-Carleson measure for $A_{\varphi}^{p}$;\\
(c) $\widetilde{\mu}\in L^{\frac{p}{p-q}}$;\\
(d) $\widehat{\mu}_{\delta}\in L^{\frac{p}{p-q}}$ for some (or any) $\delta\in(0, \alpha]$ small enough;\\
(e) $\left\{\widehat{\mu}_{r}(a_{k})\rho(a_{k})^{2-2q/p}\right\}_{k}\in l^{\frac{p}{p-q}}$ for some (or any) $(\rho, r)$-lattice $\{a_{k}\}$ with $r\in(0, \alpha]$.\\
Moreover, we have
\begin{align}\label{carleson2}
\|Id\|^{q}_{A_{\varphi}^{p}\rightarrow A_{\varphi}^{q}}
\asymp\|\widetilde{\mu}\|_{L^{\frac{p}{p-q}}}
\asymp\|\widehat{\mu}_{\delta}\|_{L^{\frac{p}{p-q}}}
\asymp\left\|\left\{\widehat{\mu}_{r}(a_{k})
\rho(a_{k})^{2-2q/p}\right\}_{k}\right\|_{l^{\frac{p}{p-q}}}.
\end{align}
\end{theorem}
\begin{proof}
The equivalence among the statements $(c)$, $(d)$ and $(e)$ comes from Proposition \ref{ThreeEquivalents} with
 \begin{align}\label{8787}
\|\widetilde{\mu}\|_{L^{\frac{p}{p-q}}}
\asymp\|\widehat{\mu}_{\delta}\|_{L^{\frac{p}{p-q}}}
\asymp\left\|\left\{\widehat{\mu}_{r}(a_{k})
\rho(a_{k})^{2-2q/p}\right\}_{k}\right\|_{l^{\frac{p}{p-q}}}
\end{align}
for some  $\delta\in(0, \alpha]$ and some $(\rho, r)$-lattice $\{a_{k}\}$ with $r\in(0, \alpha]$. The implication $(a)\Rightarrow(b)$ is trivial. It is enough to show implications $(b)\Rightarrow(e)$ and $(d)\Rightarrow(a)$.

In order to prove $(b)\Rightarrow(e)$, we will borrow the idea from Luecking \cite{Luecking1}. Assume that  $\mu$ is a $q$-Carleson measure for $A_{\varphi}^{p}$. Let $\{w_{k}\}$ be a $(\rho, r_{0})$-lattice with $r_{0}$ in (\ref{K}), for any sequence $\{c_{k}\}\in l^{p}$, where $0<p<\infty$, it follows by the boundedness of $Id$ and Proposition \ref{atomicdecomposition} that
\begin{align*}
\left\|\sum_{k}c_{k}k_{p,w_{k}}(z)\right\|_{L_{\mu}^{q}}^{q}
\lesssim\|Id\|^{q}_{A_{\varphi}^{p}\rightarrow A_{\varphi}^{q}}\left(\sum_{k}|c_{k}|^{p}\right)^{q/p}.
\end{align*}

Replacing $c_{k}$ by $c_{k}r_{k}(t)$ and integrating with respect to $t$ from 0 to 1 in the above, we arrive at
\begin{align}\label{Khinchine0}
\int_{\mathbb{D}}\int_{0}^{1}\left|\sum_{k}c_{k}r_{k}(t)
k_{p,w_{k}}(z)e^{-\varphi(z)}\right|^{q}dtd\mu(z)
\lesssim\|Id\|^{q}_{A_{\varphi}^{p}\rightarrow A_{\varphi}^{q}}\left(\sum_{k}|c_{k}|^{p}\right)^{q/p},\ 0<t<1,
\end{align}
where $r_{k}$ denotes the $k$th Rademacher function (see \cite{Luecking1, Duren1}).
Apply Fubini's theorem and Khinchine's inequality (see \cite{Luecking1}), we obtain
\begin{align}\label{Khinchine}
\int_{\mathbb{D}}\int_{0}^{1}\left|\sum_{k}c_{k}r_{k}(t)
k_{p,w_{k}}(z)e^{-\varphi(z)}\right|^{q}dtd\mu(z)
\nonumber&\gtrsim\int_{\mathbb{D}}\left(\sum_{k}|c_{k}|^{2}
\left|k_{p,w_{k}}(z)e^{-\varphi(z)}\right|^{2}\right)^{\frac{q}{2}}d\mu(z)\\
&\gtrsim\sum_{k}|c_{k}|^{q}\int_{D^{r_{0}}(w_{k})}
\left|k_{p,w_{k}}(z)e^{-\varphi(z)}\right|^{q}d\mu(z),
\end{align}
where in the last inequality we use the fact that each $z\in\mathbb{D}$ belongs to at most $N=N(r_{0})$ of the disc $D^{r_{0}}(w_{k})$.

According to estimates (\ref{ke21})  and  (\ref{K}), we deduce
\begin{align*}
\sum_{k}|c_{k}|^{q}\int_{D^{r_{0}}(w_{k})}\left|k_{p,w_{k}}(z)
e^{-\varphi(z)}\right|^{q}d\mu(z)
&\asymp\sum_{k}|c_{k}|^{q}e^{-q\varphi(w_{k})}\rho(w_{k})^{2q-\frac{2q}{p}}
\int_{D^{r_{0}}(w_{k})}\left|K(z,w_{k})e^{-\varphi(z)}\right|^{q}d\mu(z)\\
&\asymp|c_{k}|^{q}
\widehat{\mu}_{r_{0}}(w_{k})\rho(w_{k})^{2-2q/p}.
\end{align*}
This combines with estimates (\ref{Khinchine0}) and (\ref{Khinchine}) gives
\begin{align*}
\sum_{k}|c_{k}|^{q}
\widehat{\mu}_{r_{0}}(w_{k})\rho(w_{k})^{2-2q/p}
&\lesssim \|Id\|^{q}_{A_{\varphi}^{p}\rightarrow A_{\varphi}^{q}}\left(\sum_{k}|c_{k}|^{p}\right)^{q/p}.
\end{align*}
Let $b_{k}=|c_{k}|^{q}$ for each $k$, then
\begin{align*}
\sum_{k}b_{k}
\widehat{\mu}_{r_{0}}(w_{k})\rho(w_{k})^{2-2q/p}
&\lesssim\|Id\|^{q}_{A_{\varphi}^{p}\rightarrow A_{\varphi}^{q}} \left(\sum_{k}b_{k}^{p/q}\right)^{q/p}.
\end{align*}

Since the sequence $\{c_{k}\}\in l^{p}$, we see that $\{b_{k}\}\in l^{p/q}$. Note that $p/q>1$, by using the classical duality $(l^{p/q})^{\ast}\cong l^{p/p-q}$, we deduce that the sequence
$\left\{\widehat{\mu}_{r_{0}}(w_{k})\rho(w_{k})^{2-2q/p}\right\}_{k}$
belongs to $l^{p/p-q}$. By estimate (\ref{meanfuncnorm}), for any $(\rho, r)$-lattice $\{a_{k}\}$ with $r\in(0, \alpha]$, we have
\begin{align}\label{duality1}
\left\|\left\{\widehat{\mu}_{r}(a_{k})
\rho(a_{k})^{2-2q/p}\right\}_{k}\right\|_{l^{\frac{p}{p-q}}}
\asymp\left\|\left\{\widehat{\mu}_{r_{0}}(w_{k})
\rho(w_{k})^{2-2q/p}\right\}_{k}\right\|_{l^{\frac{p}{p-q}}}
\lesssim\|Id\|^{q}_{A_{\varphi}^{p}\rightarrow A_{\varphi}^{q}} .
\end{align}

Finally, we will prove $(d)\Rightarrow(a)$. Suppose that $\widehat{\mu}_{\delta}\in L^{\frac{p}{p-q}}$ for some $\delta\in(0, \alpha]$ small enough, we will show $Id:A_{\varphi}^{p}\rightarrow L_{\mu}^{q}$ is compact, i.e., if $\{f_{n}\}$ is a bounded sequence in $A_{\varphi}^{p}$ that converges to 0 uniformly on compact subsets of $\mathbb{D}$, then $\lim_{n\rightarrow\infty}\|f_{n}\|_{L_{\mu}^{q}}=0$.

Since $\rho\in\mathcal{L}$, there exists a positive constant $s$ such that $\rho(z)\leq s(1-|z|)$. Fix $\delta<\frac{1}{2s}$, then for any $r>\frac{1}{3}$, we get
\begin{align}\label{subset}
D^{\frac{\delta}{2}}(z)\subset\left\{w\in\mathbb{D}:\ |w|>\frac{r}{2}\right\},\ \text{if} \ |z|>r.
\end{align}
It follows from Lemma \ref{submeanlemma} that
\begin{align}\label{submean1}
\left|f_{n}(z)e^{-\varphi(z)}\right|^{q}\lesssim \frac{1}{\rho(z)^{2}}\int_{D^{\frac{\delta}{2}}(z)}|f_{n}(w)e^{-\varphi(w)}|^{q}dA(w).
\end{align}
Integrating respect to $d\mu$ in (\ref{submean1}), applying Fubini's theorem, (\ref{subset}) and \cite[Lemma 3.1]{Hu1}, we obtain
\begin{align}\label{estimate}
\int_{\{z\in\mathbb{D}:\ |z|>r\}}\left|f_{n}(z)e^{-\varphi(z)}\right|^{q}d\mu(z)\nonumber
&\lesssim\int_{\{z\in\mathbb{D}:\ |z|>r\}}\frac{1}{\rho(z)^{2}}
\int_{D^{\frac{\delta}{2}}(z)}\left|f_{n}(w)e^{-\varphi(w)}\right|^{q}dA(w)d\mu(z)\\
\nonumber&\lesssim\int_{\{w\in\mathbb{D}:\ |w|>\frac{r}{2}\}}\left|f_{n}(w)e^{-\varphi(w)}\right|^{q}
\left(\int_{D^{\delta}(w)}\frac{1}{\rho(z)^{2}}d\mu(z)\right)dA(w)\\
&\lesssim\int_{\{w\in\mathbb{D}:\ |w|>\frac{r}{2}\}}\left|f_{n}(w)e^{-\varphi(w)}\right|^{q}
\widehat{\mu}_{\delta}(w)dA(w).
\end{align}
Since condition $(d)$ holds, for any fixed $\varepsilon>0$, there exists $r_{0}\in(0, 1)$ such that
\begin{align*}
\int_{\{w\in\mathbb{D}:\ |w|>\frac{r_{0}}{2}\}}\widehat{\mu}_{\delta}(w)^{\frac{p}{p-q}}dA(w)
<\varepsilon^{\frac{p}{p-q}}.
\end{align*}
Since $p/q>1$, applying H\"{o}lder's inequality in estimate (\ref{estimate}), we deduce
\begin{align}\label{estimate1}
\int_{\{z\in\mathbb{D}:\ |z|>r_{0}\}}\left|f_{n}(z)e^{-\varphi(z)}\right|^{q}d\mu(z)
\lesssim\|f_{n}\|_{A_{\varphi}^{p}}^{q}\left(\int_{\{w\in\mathbb{D}:\ |w|>\frac{r_{0}}{2}\}}\widehat{\mu}_{\delta}(w)^{\frac{p}{p-q}}
dA(w)\right)^{\frac{p-q}{p}}\lesssim \varepsilon.
\end{align}
On the other hand, since $\{|z|\leq r_{0}\}$ is a compact subset of $\mathbb{D}$, we have
$$\lim_{n\rightarrow\infty}\int_{|z|\leq r_{0}}\left|f_{n}(z)e^{-\varphi(z)}\right|^{q}d\mu(z)=0.$$
 This together with (\ref{estimate1}) yields
$$\lim_{n\rightarrow\infty}\|f_{n}\|_{L_{\mu}^{q}}=0,$$
which proves $(a)$. This completes the proof of Theorem \ref{carlesonB}.
\end{proof}
\section{Bounded and compact Toeplitz Operators}\label{section4}
The next two propositions are of great importance in the proof of the compactness of Toeplitz operator $T_{\mu}$ from $A^{p}_{\varphi}$ to $A^{q}_{\varphi}$.
\begin{proposition}\label{relativelycompact}
Let $0<p<\infty$. A bounded subset $E\subset A^{p}_{\varphi}$ is relatively compact if and only if for any $\varepsilon>0$, there is some $t\in(0, 1)$ such that
\begin{align}\label{relativelycompactformula}
\sup_{f\in E}\int_{t\leq|z|<1}\left|f(z)e^{-\varphi(z)}\right|^{p}dA(z)<\varepsilon.
\end{align}
\end{proposition}
\begin{proof}
Similar to the proof of \cite[Lemma 3.2]{Hu3}
\end{proof}

\begin{proposition}\label{welldefined}
Let $\mu$ be a positive Borel measure on $\mathbb{D}$ satisfying $\widehat{\mu}_{\delta}\rho^{t}\in L^{\infty}$ for some $\delta\in(0, \alpha]$ and $t\in\mathbb{R}$. Then $T_{\mu}$ is well-defined on $A^{p}_{\varphi}$ for $0<p<\infty$. Further, for $R\in(0,1)$, the Toeplitz operator $T_{\mu_{R}}$ is compact from $A^{p}_{\varphi}$ to $A^{q}_{\varphi}$ for $0 < p, q < \infty$, where $\mu_{R}$ is defined to be
\begin{align}\label{muRV}
\mu_{R}(E)=\mu(E\cap\overline{D(0,R)})\ \ \text{for}\ E\subseteq\mathbb{D}\ \text{measurable}.
\end{align}
\end{proposition}
\begin{proof}
Assume that $\widehat{\mu}_{\delta}\rho^{t}\in L^{\infty}$ for some $\delta\in(0, \alpha]$ and $t\in\mathbb{R}$. In views of Lemma \ref{submeanlemma}, for $f\in A^{p}_{\varphi}$, we have
\begin{align}\label{111}
|f(z)|e^{-\varphi(z)}\lesssim\rho(z)^{-2/p}\|f\|_{A^{p}_{\varphi}},\ \ z\in\mathbb{D}.
\end{align}
Applying Proposition \ref{corollarysubmean} to the weight $2\varphi$ and the analytic function $K(\cdot, z)f(\cdot)$, we obtain
\begin{align}\label{222}
|T_{\mu}f(z)|\lesssim\int_{\mathbb{D}}|K(w,z)|
|f(w)|e^{-2\varphi(w)}\widehat{\mu}_{\delta}(w)dA(w).
\end{align}
This together with estimates (\ref{111}), (\ref{ke11}) and \cite[Corollary 3.1]{Hu1} yields that
\begin{align*}
|T_{\mu}f(z)|&\lesssim\|f\|_{A^{p}_{\varphi}}\left\|\widehat{\mu}_{\delta}
\rho^{t}\right\|_{L^{\infty}}\int_{\mathbb{D}}\rho(w)^{-t-2/p}
|K(w,z)|e^{-\varphi(w)}dA(w)\\
&\lesssim e^{\varphi(z)}\rho(z)^{-1}\left\|\widehat{\mu}_{\delta}
\rho^{t}\right\|_{L^{\infty}}\|f\|_{A^{p}_{\varphi}}\int_{\mathbb{D}}
\rho(w)^{-1-t-2/p}e^{-\sigma d_{\rho}(w,z)}dA(w)\\
&\lesssim e^{\varphi(z)}\rho(z)^{-t-2/p}\left\|\widehat{\mu}_{\delta}
\rho^{t}\right\|_{L^{\infty}}\|f\|_{A^{p}_{\varphi}}<\infty.
\end{align*}

Next we will prove the compactness of $T_{\mu_{R}}$. For any $f\in A^{p}_{\varphi}$ and $R<1$, by  estimate (\ref{111}) and estimate (\ref{ke11}), we obtain
\begin{align*}
\left|T_{\mu_{R}}f(z)\right|&\leq\int_{\mathbb{D}}|K(z,w)|
|f(w)|e^{-2\varphi(w)}d\mu_{R}(w)\\
&\leq\int_{\mathbb{D}}\chi_{\overline{D(0,R)}}(w)|K(z,w)|
|f(w)|e^{-2\varphi(w)}d\mu(w)\\
&\lesssim\int_{|w|\leq R}|K(z,w)|
|f(w)|e^{-2\varphi(w)}\widehat{\mu}_{\delta}(w)dA(w)\\
&\lesssim\left\|f\right\|_{A_{\varphi}^{p}}\int_{|w|\leq R}|K(z,w)|
e^{-\varphi(w)}\rho(w)^{-\frac{2}{p}}
\widehat{\mu}_{\delta}(w)dA(w)\\
&\lesssim \left\|f\right\|_{A_{\varphi}^{p}}
\frac{e^{\varphi(z)}}{\rho(z)}\int_{|w|\leq R}
\rho(w)^{-\frac{2}{p}-1}\widehat{\mu}_{\delta}(w)
e^{-\sigma d_{\rho}(z,w)}dA(w)\\
&\lesssim\sup_{|w|\leq R}\widehat{\mu}_{\delta}(w)\rho(w)^{-\frac{2}{p}-1}
\left\|f\right\|_{A_{\varphi}^{p}}
\frac{e^{\varphi(z)}}{\rho(z)}\int_{|w|\leq R}e^{-\sigma d_{\rho}(z,w)}dA(w)\\
&\lesssim\left\|f\right\|_{A_{\varphi}^{p}}
\frac{e^{\varphi(z)}}{\rho(z)}\int_{|w|\leq R}e^{-\sigma d_{\rho}(z,w)}dA(w).
\end{align*}

For $q>1$ and a fixed $\eta\in (0,1)$, apply H\"{o}lder's inequality and \cite[Corollary 3.1]{Hu1}, we deduce
\begin{align}\label{eta1}
\nonumber&\left(\int_{\mathbb{D}\backslash D(0,\eta)}\left|T_{\mu_{R}}f(z)e^{-\varphi(z)}\right|^{q}dA(z)\right)^{\frac{1}{q}}\\
\nonumber&\lesssim\left\|f\right\|_{A_{\varphi}^{p}}
\left(\int_{|w|\leq R}dA(w)\int_{\mathbb{D}\backslash D(0,\eta)}\frac{e^{-q\sigma d_{\rho}(z,w)}}{\rho(z)^{q}}dA(z)\right)^{\frac{1}{q}}\\
\nonumber&\lesssim\left\|f\right\|_{A_{\varphi}^{p}}
\sup_{|w|\leq R,\ \eta\leq|z|<1}e^{-\frac{\sigma}{2}d_{\rho}(z,w)}
\left(\int_{|w|\leq R}dA(w)\int_{\mathbb{D}\backslash D(0,\eta)}\frac{e^{-\frac{q}{2}\sigma d_{\rho}(z,w)}}{\rho(z)^{q}}dA(z)\right)^{\frac{1}{q}}\\
\nonumber&\lesssim\left\|f\right\|_{A_{\varphi}^{p}}
\sup_{|w|\leq R,\ \eta\leq|z|<1}e^{-\frac{\sigma}{2}d_{\rho}(z,w)}
\left(\int_{|w|\leq R}\rho(w)^{2-q}dA(w)\right)^{\frac{1}{q}}\\
&\lesssim\left\|f\right\|_{A_{\varphi}^{p}}
\sup_{|w|\leq R,\ \eta\leq|z|<1}e^{-\frac{\sigma}{2}d_{\rho}(z,w)}.
\end{align}

For $0<q<1$, using \cite[Corollary 3.1]{Hu1} again, we obtain
\begin{align}\label{eta2}
\nonumber&\left(\int_{\mathbb{D}\backslash D(0,\eta)}\left|T_{\mu_{R}}f(z)e^{-\varphi(z)}\right|^{q}dA(z)\right)^{\frac{1}{q}}\\
\nonumber&\lesssim\left\|f\right\|_{A_{\varphi}^{p}}
\sup_{|w|\leq R,\ \eta\leq|z|<1}e^{-\frac{\sigma}{2}d_{\rho}(z,w)}
\left(\int_{\mathbb{D}\backslash D(0,\eta)}\rho(z)^{-q}dA(z)\int_{|w|\leq R}
e^{-\frac{\sigma}{2} d_{\rho}(z,w)}dA(w)\right)^{\frac{1}{q}}\\
\nonumber&\lesssim\left\|f\right\|_{A_{\varphi}^{p}}
\sup_{|w|\leq R,\ \eta\leq|z|<1}e^{-\frac{\sigma}{2}d_{\rho}(z,w)}
\left(\int_{\mathbb{D}\backslash D(0,\eta)}\rho(z)^{2-q}dA(z)\right)^{\frac{1}{q}}\\
&\lesssim\left\|f\right\|_{A_{\varphi}^{p}}
\sup_{|w|\leq R,\ \eta\leq|z|<1}e^{-\frac{\sigma}{2}d_{\rho}(z,w)}.
\end{align}

It follows by \cite[(23)]{Hu1}, we have
$$\sup_{|w|\leq R,\ \eta\leq|z|<1}e^{-\frac{\sigma}{2}d_{\rho}(z,w)}\rightarrow0,\ \text{as}\ \eta\rightarrow1^{-}.$$

Therefore, $T_{\mu_{R}}$ is compact from $A^{p}_{\varphi}$ to $A^{q}_{\varphi}$ in views of (\ref{eta1}) and (\ref{eta2}).  This completes
the proof.
\end{proof}
Now, we embark on describing those $\mu \geq0$ for which the
induced Toeplitz operators $T_{\mu}$ are bounded (or compact) from one Bergman
space $A_{\varphi}^{p}$ to another $A_{\varphi}^{p}$, where $0 < p,q <\infty$.
\begin{theorem}\label{Toeplitzoperator}
Let $0<p\leq q<\infty$, $\varphi\in\mathcal{W}_{0}$ with $\frac{1}{\sqrt{\Delta\varphi}}\asymp\rho\in\mathcal{L}_{0}$ and $\mu$ be a finite positive Borel measure on $\mathbb{D}$. Then the following statements are equivalent:\\
(a) $T_{\mu}: A_{\varphi}^{p}\rightarrow A_{\varphi}^{q}$ is bounded;\\
(b) $\widetilde{\mu}\rho^{2(p-q)/(pq)}\in L^{\infty}$;\\
(c) $\widehat{\mu}_{\delta}\rho^{2(p-q)/(pq)}\in L^{\infty}$ for some (or any) $\delta\in(0, \alpha]$;\\
(d) $\left\{\widehat{\mu}_{r}(a_{k})\rho(a_{k})^{2(p-q)/(pq)}\right\}_{k}\in l^{\infty}$ for some (or any) $(\rho, r)$-lattice $\{a_{k}\}$ with $r\in(0, \alpha]$.\\
Moreover, we have
\begin{align}\label{ALLl}
\|T_{\mu}\|_{A_{\varphi}^{p}\rightarrow A_{\varphi}^{q}}\asymp\left\|\widetilde{\mu}\rho^{2(p-q)/(pq)}\right\|_{L^{\infty}}
\asymp\left\|\widehat{\mu}_{\delta}\rho^{2(p-q)/(pq)}\right\|_{L^{\infty}}
\asymp\left\|\left\{\widehat{\mu}_{r}(a_{k})\rho(a_{k})^{2(p-q)/(pq)}\right\}_{k}\right\|_{l^{\infty}}.
\end{align}
\end{theorem}
\begin{proof}
The implications $(c)\Rightarrow(d)$, $(b)\Rightarrow(c)$ and $(d)\Rightarrow(b)$ can be obtained by a similar discussion as the corresponding proof of Theorem \ref{carleson}. Moreover, we have the following equivalence
\begin{align}\label{Tof1}
\left\|\widetilde{\mu}\rho^{2(p-q)/(pq)}\right\|_{L^{\infty}}
\asymp\left\|\widehat{\mu}_{\delta}\rho^{2(p-q)/(pq)}\right\|_{L^{\infty}}
\asymp\left\|\left\{\widehat{\mu}_{r}(a_{k})\rho(a_{k})^{2(p-q)/(pq)}\right\}_{k}\right\|_{l^{\infty}}
\end{align}
for some  $\delta\in(0, \alpha]$ and some $(\rho, r)$-lattice $\{a_{k}\}$ with $r\in(0, \alpha]$.

$(a)\Rightarrow(b)$. In views of estimate (\ref{ke21}) and Lemma \ref{submeanlemma}, we obtain
\begin{align}\label{666}
\widetilde{\mu}(z) \rho(z)^{2(p-q) /(p q)}
\nonumber&=\int_{\mathbb{D}}|k_{z}(w)|^{2}e^{-2\varphi(w)}d\mu(w)\rho(z)^{2(p-q) /(p q)}\\
\nonumber&\asymp\rho(z)^{2/q}\int_{\mathbb{D}}k_{p, z}(w)K(z,w)e^{-2\varphi(w)}d\mu(w)e^{-\varphi(z)}\\
\nonumber&\leq\rho(z)^{2/q}|T_{\mu}k_{p,z}(z)|e^{-\varphi(z)}\\
&\lesssim\left(\int_{D^{\frac{\delta}{2}}(z)}|T_{\mu}k_{p,z}(z)e^{-\varphi(z)}|^{q}dA(w)\right)^{1/q},
\end{align}
where $\delta\in(0, \alpha]$. Hence
\begin{align}\label{LAA}
\left\|\widetilde{\mu}(z) \rho(z)^{2(p-q) /(p q)}\right\|_{L^{\infty}}\lesssim\left\|T_{\mu}k_{p,z}(z)\right\|_{A_{\varphi}^{q}}
\lesssim\left\|T_{\mu}\right\|_{A_{\varphi}^{p}\rightarrow A_{\varphi}^{q}},
\end{align}
which gives $(a)\Rightarrow(b)$.

$(c)\Rightarrow(a)$. Assume that $\widehat{\mu}_{\delta}\rho^{2(p-q)/(pq)}$ is in $L^{\infty}$ for some $\delta\in(0, \alpha]$, then $T_{\mu}$ is well-defined on $A_{\varphi}^{p}$  according to Proposition \ref{welldefined}. We assert that
\begin{align}\label{333}
\|T_{\mu}f\|_{A_{\varphi}^{q}}^{q}\lesssim\int_{\mathbb{D}}|f(w)|^{q}e^{-q\varphi(w)}
\widehat{\mu}_{\delta}(w)^{q}dA(w)
\end{align}
for $f\in A_{\varphi}^{p}$. In fact, if $q>1$, (\ref{222}), H\"{o}lder's inequality and estimate (\ref{ke21}) imply that
\begin{align*}
\left|T_{\mu}f(z)\right|^{q}e^{-q\varphi(z)}&\lesssim\left(\int_{\mathbb{D}}
|f(w)|\widehat{\mu}_{\delta}(w)|K(w,z)|e^{-2\varphi(w)}e^{-\varphi(z)}dA(w)\right)^{q}\\
&\leq\int_{\mathbb{D}}|f(w)|^{q}e^{-q\varphi(w)}\widehat{\mu}_{\delta}(w)^{q}\left|K(w,z)
e^{-\varphi(w)}e^{-\varphi(z)}\right|dA(w)\\
&\hs\hs\cdot\left(\int_{\mathbb{D}}\left|K(w,z)e^{-\varphi(w)}e^{-\varphi(z)}
\right|dA(w)\right)^{q/q^{\prime}}\\
&\lesssim\int_{\mathbb{D}}|f(w)|^{q}e^{-q\varphi(w)}\widehat{\mu}_{\delta}(w)^{q}\left|K(w,z)
e^{-\varphi(w)}e^{-\varphi(z)}\right|dA(w).
\end{align*}
Thus, by Fubini's Theorem and estimate (\ref{ke21}), we have
\begin{align*}
\int_{\mathbb{D}}|T_{\mu}f(z)|^{q}e^{-q\varphi(z)}dA(z)
&\lesssim\int_{\mathbb{D}}|f(w)|^{q}e^{-q\varphi(w)}
\widehat{\mu}_{\delta}(w)^{q}dA(w)\int_{\mathbb{D}}
\left|K(w,z)e^{-\varphi(w)}e^{-\varphi(z)}\right|dA(z)\\
&\lesssim\int_{\mathbb{D}}|f(w)|^{q}e^{-q\varphi(w)}
\widehat{\mu}_{\delta}(w)^{q}dA(w),
\end{align*}
which gives (\ref{333}) in the case of $q>1$.

If $0<q\leq1$, for the fixed $\delta\in(0, \alpha]$, we choose some $r>0$ small enough such that $B^{2}r\leq\min\{\delta,1\}$ and $Br<2\alpha$, where the constant $B$ is chosen from \cite[Lemma 3.1 (B)]{Hu1}. Let $\{a_{k}\}$ be some $(\rho,r)$-lattice, for any $f\in A_{\varphi}^{p}$, Lemma \ref{latticelemma}, (\ref{fact0}), Lemma \ref{submeanlemma} and \cite[Lemma 3.1]{Hu1} tell us
\begin{align*}
|T_{\mu}f(z)|^{q}&\leq\left(\sum_{k=1}^{\infty}\int_{D^{r}(a_{k})}|f(w)K(w,z)|
e^{-2\varphi(w)}d\mu(w)\right)^{q}\\
&\leq\sum_{k=1}^{\infty}\left(\int_{D^{r}(a_{k})}|f(w)K(w,z)|
e^{-2\varphi(w)}d\mu(w)\right)^{q}\\
&\leq\sum_{k=1}^{\infty}\widehat{\mu}_{r}(a_{k})^{q}\rho(a_{k})^{2q}
\left(\sup_{w\in D^{r}(a_{k})}|f(w)K(w,z)|e^{-2\varphi(w)}\right)^{q}\\
&\lesssim\sum_{k=1}^{\infty}\widehat{\mu}_{r}(a_{k})^{q}\rho(a_{k})^{2q-2}
\int_{D^{Br}(a_{k})}|f(w)|^{q}|K(w,z)|^{q}e^{-2q\varphi(w)}dA(w).
\end{align*}

By \cite[Lemma 3.1 (B)]{Hu1}, we see that $D^{r}(a_{k})\subseteq D^{B^{2}r}(w)$ if $w\in D^{Br}(a_{k})$. This together with \cite[Lemma 3.1]{Hu1} and Lemma \ref{latticelemma} gives
\begin{align*}
|T_{\mu}f(z)|^{q}&\lesssim\sum_{k=1}^{\infty}\int_{D^{2\alpha}(a_{k})}
\widehat{\mu}_{B^{2}r}(w)^{q}\rho(w)^{2q-2}|f(w)|^{q}
|K(w,z)|^{q}e^{-2q\varphi(w)}dA(w)\\
&\lesssim N\int_{\mathbb{D}}\widehat{\mu}_{B^{2}r}(w)^{q}\rho(w)^{2q-2}|f(w)|^{q}
|K(w,z)|^{q}e^{-2q\varphi(w)}dA(w)\\
&\lesssim\int_{\mathbb{D}}\widehat{\mu}_{\delta}(w)^{q}\rho(w)^{2q-2}|f(w)|^{q}
|K(w,z)|^{q}e^{-2q\varphi(w)}dA(w).
\end{align*}
Integrating both sides above with respect to $e^{-q\varphi(z)}dA(z)$, applying Fubini's theorem and using estimate (\ref{ke21}), we obtain assertion (\ref{333}).

Hence, it follows by estimates (\ref{333}) and (\ref{111}) that
\begin{align*}
\|T_{\mu}f\|_{A_{\varphi}^{q}}^{q}\lesssim\int_{\mathbb{D}}|f(w)|^{p}e^{-p\varphi(z)}\widehat{\mu}_{\delta}(w)^{q}
\left(\rho(w)^{-2/p}\|f\|_{A_{\varphi}^{p}}\right)^{q-p}dA(w)
\lesssim\left\|\widehat{\mu}_{\delta}\rho^{2(p-q) /(p q)}\right\|\|f\|_{A_{\varphi}^{p}}^{q}
\end{align*}
for $f\in A_{\varphi}^{p}$.
Therefore, $T_{\mu}$ is bounded from $A_{\varphi}^{p}$ to $A_{\varphi}^{q}$ with
\begin{align}\label{AL}
\|T_{\mu}\|_{A_{\varphi}^{p}\rightarrow A_{\varphi}^{q}}\lesssim\left\|\widehat{\mu}_{\delta} \rho^{2(p-q) /(p q)}\right\|_{L^{\infty}},
\end{align}
which gives $(c)\Rightarrow(a)$. The equivalence of  (\ref{ALLl}) follows from (\ref{Tof1}), (\ref{LAA}) and (\ref{AL}). This completes the proof.
\end{proof}
\begin{theorem}\label{ToeplitzoperatorA}
Let $0<p\leq q<\infty$, $\varphi\in\mathcal{W}_{0}$ with $\frac{1}{\sqrt{\Delta\varphi}}\asymp\rho\in\mathcal{L}_{0}$ and $\mu$ be a finite positive Borel measure on $\mathbb{D}$. Then the following statements are equivalent:\\
(a) $T_{\mu}: A_{\varphi}^{p}\rightarrow A_{\varphi}^{q}$ is compact;\\
(b) $\lim_{|z|\rightarrow1^{-}}\widetilde{\mu}(z)\rho(z)^{2(p-q)/(pq)}=0$;\\
(c) $\lim_{|z|\rightarrow1^{-}}\widehat{\mu}_{\delta}(z)\rho(z)^{2(p-q)/(pq)}=0$ for some (or any) $\delta\in(0, \alpha]$;\\
(d) $\lim_{k\rightarrow\infty}\widehat{\mu}_{r}(a_{k})\rho(a_{k})^{2(p-q)/(pq)}=0$  for some (or any) $(\rho, r)$-lattice $\{a_{k}\}$ with $r\in(0, \alpha]$.
\end{theorem}
\begin{proof}
The proof of the implications $(b)\Rightarrow(c)$ and $(c)\Rightarrow(d)$ are similar to the same part of Theorem \ref{carleson}, and the proof of $(d)\Rightarrow(b)$ is similar to that of Theorem \ref{carlesonA}.

$(a)\Rightarrow(b)$. Assume that $T_{\mu}$ is compact from $A_{\varphi}^{p}$ to $A_{\varphi}^{q}$. Since the set $\left\{k_{p,z}:z\in\mathbb{D}\right\}$ is bounded in $A_{\varphi}^{p}$, $\left\{T_{\mu}k_{p,z}:z\in\mathbb{D}\right\}$ is relatively compact in $A_{\varphi}^{q}$. For any $\varepsilon>0$, according to (\ref{relativelycompactformula}),
 there is some $S > 0$ such that
\begin{align*}
\sup_{z\in\mathbb{D}}\int_{|w|> \frac{S}{2}}\left|T_{\mu}k_{p,z}(w)e^{-\varphi(w)}\right|^{q}dA(w)<\varepsilon^{q}.
\end{align*}
Then by (\ref{subset}), we have
$D^{\frac{\delta}{2}}(z)\subset\left\{w\in\mathbb{D}:\ |w|>\frac{S}{2}\right\}$ when $|z|>S$ and $\delta\in(0, \alpha]$  small enough. Hence, by estimate (\ref{666}), we obtain
\begin{align*}
\widetilde{\mu}(z) \rho(z)^{2(p-q) /(p q)}
\lesssim\left(\int_{D^{\frac{\delta}{2}}(z)}\left|T_{\mu}k_{p,z}(z)
e^{-\varphi(z)}\right|^{q}dA(w)\right)^{1/q}\lesssim\varepsilon
\end{align*}
as $|z|$ is sufficiently close to $1^{-}$. Thus we get
$$
\lim _{|z| \rightarrow 1^{-}} \widetilde{\mu}(z) \rho(z)^{2(p-q) /(p q)}=0.
$$

$(c)\Rightarrow(a)$. If $\lim_{|z|\rightarrow1^{-}}\widehat{\mu}_{\delta}(z)\rho(z)^{2(p-q)/(pq)}=0$ for some $\delta\in(0, \alpha]$. It follows by Proposition \ref{welldefined} that $T_{\mu_{R}}$ is compact from $A^{p}_{\varphi}$ to $A^{q}_{\varphi}$. Since $\mu-\mu_{R}\geq0$, in the light of (\ref{ALLl}) and the assumption,  we have
$$
\left\|T_{\mu}-T_{\mu_{R}}\right\|_{A^{p}_{\varphi}\rightarrow A^{q}_{\varphi}} \asymp\left\|\left(\widehat{\mu-\mu_{R}}\right)_{\delta} \rho^{2(p-q) /(p q)}\right\|_{L^{\infty}} \rightarrow 0
$$
as $R\rightarrow1^{-}$. Therefore, $T_{\mu}$ is compact from $A^{p}_{\varphi}$ to $A^{q}_{\varphi}$. This completes the proof.
\end{proof}
\begin{theorem}\label{ToeplitzoperatorB}
Let $0<q<p<\infty$, $\varphi\in\mathcal{W}_{0}$ with $\frac{1}{\sqrt{\Delta\varphi}}\asymp\rho\in\mathcal{L}_{0}$ and $\mu$ be a finite positive Borel measure on $\mathbb{D}$. Then the following statements are equivalent:\\
(a) $T_{\mu}: A_{\varphi}^{p}\rightarrow A_{\varphi}^{q}$ is compact;\\
(b) $T_{\mu}: A_{\varphi}^{p}\rightarrow A_{\varphi}^{q}$ is bounded;\\
(c) $\widetilde{\mu}\in L^{\frac{pq}{p-q}}$;\\
(d) $\widehat{\mu}_{\delta}\in L^{\frac{pq}{p-q}}$ for some (or any) $\delta\in(0, \alpha]$;\\
(e) $\left\{\widehat{\mu}_{r}(a_{k})\rho(a_{k})^{2(p-q)/(pq)}\right\}_{k}\in l^{\frac{pq}{p-q}}$ for some (or any) $(\rho, r)$-lattice $\{a_{k}\}$ with $r\in(0, \alpha]$.\\
Moreover, we have
\begin{align}\label{ALLl1}
\|T_{\mu}\|_{A_{\varphi}^{p}\rightarrow A_{\varphi}^{q}}
\asymp\|\widetilde{\mu}\|_{L^{\frac{pq}{p-q}}}
\asymp\|\widehat{\mu}_{\delta}\|_{L^{\frac{pq}{p-q}}}
\asymp\left\|\left\{\widehat{\mu}_{r}(a_{k})
\rho(a_{k})^{2(p-q)/(pq)}\right\}_{k}\right\|_{l^{\frac{pq}{p-q}}}.
\end{align}
\end{theorem}
\begin{proof}
The equivalence among the conclusions $(c)$, $(d)$ and $(e)$ follows directly from Proposition \ref{ThreeEquivalents} with
\begin{align}\label{ALLl111}
\|\widetilde{\mu}\|_{L^{\frac{pq}{p-q}}}
\asymp\|\widehat{\mu}_{\delta}\|_{L^{\frac{pq}{p-q}}}
\asymp\left\|\left\{\widehat{\mu}_{r}(a_{k})
\rho(a_{k})^{2(p-q)/(pq)}\right\}_{k}\right\|_{l^{\frac{pq}{p-q}}}
\end{align}
for some  $\delta\in(0, \alpha]$ and some $(\rho, r)$-lattice $\{a_{k}\}$ with $r\in(0, \alpha]$. The implication $(a)\Rightarrow(b)$ is obvious. And the proof of implication $(b)\Rightarrow(e)$ is similar to the corresponding part of Theorem \ref{carlesonB} with
\begin{align}\label{hatformula}
\left\|\left\{\widehat{\mu}_{r}(a_{k})
\rho(a_{k})^{2(p-q)/(pq)}\right\}_{k}\right\|_{l^{\frac{pq}{p-q}}}
\lesssim\|T_{\mu}\|_{A_{\varphi}^{p}\rightarrow A_{\varphi}^{q}}
\end{align}
for any $(\rho,r)$-lattice $\{a_{k}\}$ with $r\in(0, \alpha]$.
It suffices to show the implication $(d)\Rightarrow(a)$ to complete our proof.

$(d)\Rightarrow(a)$. Assume that $\widehat{\mu}_{\delta}\in L^{\frac{pq}{p-q}}$ for some $\delta\in(0, \alpha]$, it follows from Proposition \ref{ThreeEquivalents} that the sequence $\left\{\widehat{\mu}_{\delta}(a_{k})\rho(a_{k})^{2(p-q)/(pq)}\right\}_{k}\in l^{\infty}$ for some $(\rho,\delta)$-lattice $\{a_{k}\}$. According to Theorem \ref{Toeplitzoperator}, we see that $\widehat{\mu}_{\delta}\rho^{2(p-q)/(pq)}\in L^{\infty}$. Hence $T_{\mu}$ is well-defined on $A_{\varphi}^{p}$ by Proposition \ref{welldefined}.

Since $p/q>1$, bearing in mind (\ref{333}), applying H\"{o}lder's inequality and Lemma \ref{kernelestimate2}, we have
\begin{align*}
\|T_{\mu}f\|_{A_{\varphi}^{q}}^{q}&\lesssim\left\{\int_{\mathbb{D}}\left(
|f(w)|^{q}e^{-q\varphi(w)}\right)^{p/q}dA(w)\right\}^{q/p}
\left\{\int_{\mathbb{D}}\widehat{\mu}_{\delta}(w)^{pq/(p-q)}dA(w)\right\}^{(p-q)/p}\lesssim\|\widehat{\mu}_{\delta}\|^{q}_{L^{\frac{pq}{p-q}}}\|f\|_{A_{\varphi}^{p}}^{q}
\end{align*}
for $f\in A_{\varphi}^{p}$. Therefore, $T_{\mu}$ is bounded from $A_{\varphi}^{p}$ to $A_{\varphi}^{q}$ with
\begin{align}\label{444}
\|T_{\mu}\|_{A_{\varphi}^{p}\rightarrow A_{\varphi}^{q}}
\lesssim\|\widehat{\mu}_{\delta}\|_{L^{\frac{pq}{p-q}}}.
\end{align}

To prove the compactness of $T_{\mu}$ from $A_{\varphi}^{p}$ to $A_{\varphi}^{q}$, we take $\mu_{R}$ as in (\ref{muRV}), then $\mu-\mu_{R} \geq 0$, and we have $\left\|\left(\widehat{\mu-\mu_{R}}\right)_{\delta}\right\|_{L^{\frac{pq}{p-q}}} \rightarrow 0$ as $R\rightarrow1^{-}$. By estimate (\ref{444}), we obtain
$$
\left\|T_{\mu}-T_{\mu_{R}}\right\|_{A_{\varphi}^{p}\rightarrow A_{\varphi}^{q}}=\left\|T_{\left(\mu-\mu_{R}\right)}\right\|_{A_{\varphi}^{p}\rightarrow A_{\varphi}^{q}} \lesssim\left\|\left(\widehat{\mu-\mu_{R}}\right)_{\delta}\right\|_{L^{\frac{pq}{p-q}}} \rightarrow 0
$$
whenever $R\rightarrow1^{-}$. Thus $T_{\mu}$ is compact from $A_{\varphi}^{p}$ to $A_{\varphi}^{q}$, since $T_{\mu_{R}}$ is compact by Proposition \ref{welldefined}. The equivalence of norms (\ref{ALLl1}) follows from (\ref{ALLl111}), (\ref{hatformula}) and (\ref{444}). This completes the proof.
\end{proof}
\section{Schatten class Toeplitz  Operators}\label{section5}
In this section, we will investigate when the Toeplitz operator $T_{\mu}$ belongs to the Schatten $p$-class $\mathcal{S}_{p}(A_{\varphi}^{2})$. We refer to \cite[Chapter 1]{Zhu1} for more information about the Schatten $p$-class. We first present the following technical result, which is an easy consequence of Fubini's theorem and the reproducing property.
\begin{lemma}\label{classicalequation}
Let $\varphi\in\mathcal{W}_{0}$ with $\frac{1}{\sqrt{\Delta\varphi}}\asymp\rho\in\mathcal{L}_{0}$ and $\mu$ be a positive Borel measure on $\mathbb{D}$. Then
$$\langle T_{\mu}f,g\rangle_{A_{\varphi}^{2}}=\int_{\mathbb{D}}f(w)\overline{g(w)}e^{-2\varphi(w)}d\mu(w),\quad f,g\in A_{\varphi}^{2}.$$
\end{lemma}

Next, we will consider the measure $d\lambda_{\rho}$ given by
$$d\lambda_{\rho}(z)=\frac{dA(z)}{\rho(z)^{2}},\ \ z\in\mathbb{D}.$$
\begin{proposition}\label{ThreeEquivalents1}
Let $0<p<\infty$ and $\mu$ be a positive Borel measure on $\mathbb{D}$. Then the following statements are equivalent:\\
(a) $\widetilde{\mu}\in L^{p}(\mathbb{D}, d\lambda_{\rho})$;\\
(b) $\widehat{\mu}_{\delta}\in L^{p}(\mathbb{D}, d\lambda_{\rho})$ for some (or any) $\delta\in(0, \alpha]$;\\
(c) The sequence $\left\{\widehat{\mu}_{r}(w_{k})\right\}_{k}\in l^{p}$ for some (or any) $(\rho, r)$-lattice $\{w_{k}\}$ with $r\in(0, \alpha]$.
\end{proposition}
\begin{proof}
Similar to that of Proposition \ref{ThreeEquivalents}.
\end{proof}
\begin{lemma}\label{Schatten1}
Suppose $\varphi\in\mathcal{W}_{0}$ with $\frac{1}{\sqrt{\Delta\varphi}}\asymp\rho\in\mathcal{L}_{0}$ and $T$ is a positive operator on $A_{\varphi}^{2}$. Let $\widetilde{T}$ denote Berezin transform of the operator $T$ defined by
$$\widetilde{T}(z)=\left\langle Tk_{z},k_{z}\right\rangle_{A_{\varphi}^{2}},\ \ z\in \mathbb{D}.$$
(a) Let $0<p\leq1$. If $\widetilde{T}\in L^{p}(\mathbb{D}, d\lambda_{\rho})$, then $T$ is in $\mathcal{S}_{p}(A_{\varphi}^{2})$.\\
(b) Let $p\geq1$. If $T$ is in $\mathcal{S}_{p}(A_{\varphi}^{2})$, then $\widetilde{T}\in L^{p}(\mathbb{D}, d\lambda_{\rho})$.
\end{lemma}
\begin{proof}
If $p>0$, it is known that the positive operator $T$ is in $\mathcal{S}_{p}(A_{\varphi}^{2})$ if and only if $T^{p}$ is in $\mathcal{S}_{1}(A_{\varphi}^{2})$, and $T^{p}$ is in $\mathcal{S}_{1}(A_{\varphi}^{2})$ if and only if $\sum_{k}\left\langle T^{p} e_{k}, e_{k}\right\rangle_{A_{\varphi}^{2}}<\infty$, where $\{e_{k}\}$ is an orthonormal basis of $A_{\varphi}^{2}$. Let $S=\sqrt{T^{p}}$, then by the reproducing property, Fubini's theorem, Parseval's identity and Lemma \ref{kernelestimate2}, we obtain
\begin{align*}
\sum_{k}\left\langle T^{p} e_{k}, e_{k}\right\rangle_{A_{\varphi}^{2}}&=\sum_{k}\left\|S e_{k}\right\|_{A_{\varphi}^{2}}^{2}=
\sum_{k} \int_{\mathbb{D}}\left|S e_{k}(z)\right|^{2}e^{-2\varphi(z)}d A(z)=\sum_{k}\int_{\mathbb{D}}\left|\left\langle S e_{k}, K_{z}\right\rangle_{A_{\varphi}^{2}}\right|^{2}e^{-2\varphi(z)}d A(z)\\
&=\int_{\mathbb{D}}\sum_{k}\left|\left\langle e_{k}, S K_{z}\right\rangle_{A_{\varphi}^{2}}\right|^{2} e^{-2\varphi(z)}d A(z)
=\int_{\mathbb{D}}\left\|S K_{z}\right\|_{A_{\varphi}^{2}}^{2} e^{-2\varphi(z)} d A(z)\\
&=\int_{\mathbb{D}}\left\langle T^{p} K_{z}, K_{z}\right\rangle_{A_{\varphi}^{2}} e^{-2\varphi(z)} d A(z)
=\int_{\mathbb{D}}\left\langle T^{p} k_{z}, k_{z}\right\rangle_{A_{\varphi}^{2}}\left\|K_{z}\right\|_{A_{\varphi}^{2}}^{2} e^{-2\varphi(z)} d A(z)\\
&\asymp\int_{\mathbb{D}}\left\langle T^{p} k_{z}, k_{z}\right\rangle_{A_{\varphi}^{2}} d \lambda_{\rho}(z).
\end{align*}
Therefore, the conclusions $(a)$ and $(b)$ are obtained by using the following two inequalities (see \cite[Proposition 1.31]{Zhu1})
$$
\left\langle T^{p} k_{z}, k_{z}\right\rangle_{A_{\varphi}^{2}} \leq\left[\left\langle T k_{z}, k_{z}\right\rangle_{A_{\varphi}^{2}}\right]^{p}=[\widetilde{T}(z)]^{p}, \quad 0<p \leq 1,
$$
and
$$
[\widetilde{T}(z)]^{p}=\left[\left\langle T k_{z}, k_{z}\right\rangle_{A_{\varphi}^{2}}\right]^{p} \leq\left\langle T^{p} k_{z}, k_{z}\right\rangle_{A_{\varphi}^{2}}, \quad p \geq 1.
$$
\end{proof}
\begin{corollary}\label{Schattencorollary1}
Let $\varphi\in\mathcal{W}_{0}$ with $\frac{1}{\sqrt{\Delta\varphi}}\asymp\rho\in\mathcal{L}_{0}$. If $0<p\leq1$ and $\widetilde{\mu}\in L^{p}(\mathbb{D}, d\lambda_{\rho})$, then $T_{\mu}$ is in $\mathcal{S}_{p}(A_{\varphi}^{2})$. Conversely, if $p\geq1$ and $T_{\mu}$ is in $\mathcal{S}_{p}(A_{\varphi}^{2})$, then $\widetilde{\mu}\in L^{p}(\mathbb{D}, d\lambda_{\rho})$.
\end{corollary}
\begin{proof}
Assume that $\widetilde{\mu}\in L^p(\mathbb{D},d\lambda_{\rho})$, then  $\{\widehat{\mu}_{r}(w_{k})\}_{k}\in l^{\infty}$ for some $(\rho,r)$-lattice $\{w_{k}\}$ by Proposition \ref{ThreeEquivalents1}. Thus $T_{\mu}$ is bounded on $A_{\varphi}^{2}$ in views of Theorem \ref{Toeplitzoperator}.
 Since $\widetilde{T}_{\mu}=\widetilde{\mu}$ by a simple calculation, the result follows immediately from Lemma \ref{Schatten1}.
\end{proof}
For any $E\subseteq\mathbb{D}$ Lebesgue measurable, we write $\left|E\right|$ for the Lebesgue area measure of $E$ in the following.
\begin{lemma}\label{lll1}
Let $\rho\in\mathcal{L}_{0}$ and $\{w_{j}\}$ be a $(\rho,\delta)$-lattice on $\mathbb{D}$. For any $\xi\in\mathbb{D}$ and $k\in\mathbb{N}^{+}$, the set
\begin{align*}
D_{k}(\xi)=\left\{z\in\mathbb{D}:|\xi-z|<2^{k}\delta\min
\left(\rho(\xi),\rho(z)\right)\right\}
\end{align*}
contains at most $K$ points of the lattice $\{w_{j}\}$, where $K$ depends on $k$ but not on $\xi$.
\end{lemma}
\begin{proof}
Let $K$ denote the number of points of the lattice contained in the set $D_{k}(\xi)$. For $w_{j}\in D_{k}(\xi)$, by the definition of $\mathcal{L}_{0}$, we get
\begin{align}\label{2222}
\rho(\xi)\leq\rho(w_{j})+C|\xi-w_{j}|\leq(1+C2^{k}\delta)\rho(w_{j})=C_{k}\rho(w_{j}).
\end{align}
Since there exists some $s>0$ such that $D^{\delta s}(w_{i})\bigcap D^{\delta s}(w_{j})=\emptyset$ for $i\neq j$ by Lemma \ref{latticelemma}, we have
\begin{align}\label{3333}
K\rho(\xi)^{2}\leq C_{k}^{2}\sum_{w_{j}\in D_{k}(\xi)}\rho(w_{j})^{2}
\lesssim C_{k}^{2}\left|\bigcup_{w_{j}\in D_{k}(\xi)}D^{\delta s}(w_{j})\right|
\end{align}
As done in (\ref{2222}), we get $\rho(w_{j})\leq C_{k}\rho(\xi)$ for $w_{j}\in D_{k}(\xi)$. Hence there exists some constant $c$ such that $D^{\delta s}(w_{j})\subset D^{c2^{k}\delta}(\xi)$. Then we have
\begin{align*}
\bigcup_{w_{j}\in D_{k}(\xi)}D^{\delta s}(w_{j})\subset D^{c2^{k}\delta}(\xi).
\end{align*}
Therefore, this together with (\ref{3333}) gives
\begin{align*}
K\rho(\xi)^{2}\leq C_{k}^{2}\left|D^{c2^{k}\delta}(\xi)\right|\lesssim 2^{2k}\rho(\xi)^{2},
\end{align*}
which shows $K\lesssim2^{2k}$.
\end{proof}
\begin{lemma}\label{lll2}
Let $\rho\in\mathcal{L}_{0}$, $\delta\in(0,\alpha]$ and $k\in\mathbb{N}^{+}$. For any $(\rho,\delta)$-lattice $\{w_{j}\}$ on $\mathbb{D}$, we can divide it into $M$ subsequences which satisfies that if $w_{i}$ and $w_{j}$ are two different
points in the same subsequence, then $|w_{i}-w_{j}|\geq2^{k}\delta\min
\left(\rho(w_{i}),\rho(w_{j})\right)$.
\end{lemma}
\begin{proof}
Let $K$ denote the number given by Lemma \ref{lll1}. Consider the $(\rho,\delta)$-lattice $\{w_{j}\}$. Let $w_{j_{1}}=w_{1}$. If a point $w_{j_{m}}$ has been chosen from the lattice $\{w_{j}\}$, we choose $w_{j_{m+1}}\in\{w_{j}\}$ as the first term after $w_{j_{m}}$ such that $\left|w_{j_{m+1}}-w_{j_{l}}\right|\geq2^{k}\delta\min
\left(\rho(w_{j_{m+1}}),\rho(w_{j_{l}})\right)$, where $l=1, \ldots, m$. By induction, we can extract a $2^{k}\delta$-subsequence $\{w_{j_{m}}\}_{m}$ from the lattice $\{w_{j}\}$, we stop once the subsequence is maximal, i.e., when all the remaining points $x$ of $\{w_{j}\}$ satisfy $\left|x-w_{x}\right|<2^{k}\delta\min
\left(\rho(x),\rho(w_{x})\right)$ for some $w_{x}$ in the subsequence. We continue to select another maximal $2^{k}\delta$-subsequence from the remaining points of
$\{w_{j}\}$, and we repeat the process until we obtain $M = K + 1$ maximal $2^{k}\delta$-subsequences. If there are no points left from the lattice, things are done. If at least one point is left from $\{w_{j}\}$, this implies that there are $M = K + 1$ distinct points in the lattice satisfying $\left|w_{i}-w_{j}\right|<2^{k}\delta\min
\left(\rho(w_{i}),\rho(w_{j})\right)$, which  contradicts with the choice of $K$ from Lemma \ref{lll1}. This completes the proof of the lemma.
\end{proof}
Our characterizations on the Toeplitz operators $T_{\mu}$ in the Schatten $p$-class $\mathcal{S}_{p}(A_{\varphi}^{2})$ are shown as follows.

\begin{theorem}\label{Schattenclass}
Let $\varphi\in\mathcal{W}_{0}$ with $\frac{1}{\sqrt{\Delta\varphi}}\asymp\rho\in\mathcal{L}_{0}$, $0<p<\infty$ and $\mu$ be a finite positive Borel measure on $\mathbb{D}$. Then the following statements are equivalent:\\
(a) The Toeplitz operator $T_\mu$ is in $\mathcal{S}_{p}(A_{\varphi}^{2})$;\\
(b) The function $\widehat{\mu}_{\delta}$ is in $L^p(\mathbb{D},d\lambda_{\rho})$ for some (or any) $\delta\in(0, \alpha]$;\\
(c) The sequence $\{\widehat{\mu}_{r}(w_{n})\}_{n}\in l^{p}$ for some (or any) $(\rho, r)$-lattice $\{w_{n}\}$ with $r\in(0, \alpha]$ sufficiently small;\\
(d) $\widetilde{T}_{\mu}\in L^p(\mathbb{D},d\lambda_{\rho})$.
\end{theorem}
\begin{proof}
According to Proposition \ref{ThreeEquivalents1}, the statements $(b)$, $(c)$ and $(d)$ are equivalent. For $p>1$, we have the implication $(a)\Rightarrow(b)$ by Corollary \ref{Schattencorollary1} and the equivalence $(b)\Leftrightarrow(d)$. The implication $(c)\Rightarrow(a)$ for $0<p<1$ follows directly from Corollary \ref{Schattencorollary1} and $(c)\Leftrightarrow(d)$. To complete our proof, it remains to show the implication $(b)\Rightarrow(a)$ for $p>1$ and  $(a)\Rightarrow(c)$ for $0<p<1$.

We will first prove $(b)\Rightarrow(a)$ for $p>1$. Suppose $\widehat{\mu}_{\delta}\in L^p(\mathbb{D},d\lambda_{\rho})$ for some $\delta\in(0, \alpha]$, then $T_{\mu}$ must be compact on $A_{\varphi}^{2}$ by the equivalent $(b)\Leftrightarrow(c)$ and Theorem \ref{ToeplitzoperatorA}.
It is easy to see that
\begin{align}\label{555}
\sum_{n}\left\langle T_{\mu} e_{n}, e_{n}\right\rangle_{A_{\varphi}^{2}}^{p}=\sum_{n}\left(\int_{\mathbb{D}}\left|e_{n}(z)\right|^{2} e^{-2\varphi(z)} d \mu(z)\right)^{p}
\end{align}
for any orthonormal set $\left\{e_{n}\right\}$ on $A_{\varphi}^{2}$. It follows by Proposition \ref{corollarysubmean} that
\begin{align*}
\int_{\mathbb{D}}\left|e_{n}(z)\right|^{2} e^{-2\varphi(z)} d \mu(z)
\lesssim \int_{\mathbb{D}}\left|e_{n}(w)\right|^{2} e^{-2\varphi(w)} \widehat{\mu}_{\delta}(w) d A(w).
\end{align*}
Since $p>1$ and $\left\|e_{n}\right\|_{A_{\varphi}^{2}}=1$, apply H\"{o}lder's inequality, we arrive at
$$
\left(\int_{\mathbb{D}}\left|e_{n}(z)\right|^{2} e^{-2\varphi(z)} d \mu(z)\right)^{p} \lesssim \int_{\mathbb{D}}\left|e_{n}(w)\right|^{2} e^{-2\varphi(w)} \widehat{\mu}_{\delta}(w)^{p} d A(w).
$$
Putting the above inequality into (\ref{555}) and using estimate (\ref{ke21}), we obtain
\begin{align*}
\sum_{n}\left\langle T_{\mu} e_{n}, e_{n}\right\rangle_{A_{\varphi}^{2}}^{p}
&\lesssim\int_{\mathbb{D}}\left(\sum_{n}\left|e_{n}(w)\right|^{2}\right) e^{-2\varphi(w)} \widehat{\mu}_{\delta}(w)^{p} d A(w)\\
&\leq\int_{\mathbb{D}}\|K_{w}\|_{A_{\varphi}^{2}}^{2}e^{-2\varphi(w)} \widehat{\mu}_{\delta}(w)^{p} d A(w)\asymp\int_{\mathbb{D}}\widehat{\mu}_{\delta}(w)^{p}d\lambda_{\rho}(w).
\end{align*}
It follows from \cite[Theorem 1.27]{Zhu1} that $T_{\mu}$ is in $S_{p}(A_{\varphi}^{2})$ with $\|T_{\mu}\|_{S_{p}(A_{\varphi}^{2})} \lesssim\|\widehat{\mu}_{\delta}\|_{L^{p}\left(\mathbb{D}, d \lambda_{\rho}\right)}$.

Finally, we will show $(a)\Rightarrow(c)$ for $0<p<1$. The idea for this proof has its origins in the previous work of Semmes \cite{Semmes1} and Luecking \cite{Luecking0}. Assume that $T_\mu$ is in $\mathcal{S}_{p}(A_{\varphi}^{2})$. Let $\left\{w_{n}\right\}$ be a $(\rho,r)$-lattice on $\mathbb{D}$ with $r\in(0, \alpha]$ sufficiently small. We need to prove that the sequence $\left\{\widehat{\mu}_{r}\left(w_{n}\right)\right\}_{n}$ is in $l^{p}$. For this purpose, we fix a large positive integer $m\geq2$ and use Lemma \ref{lll2} to divide the lattice $\left\{w_{n}\right\}$ into $M$ subsequences such that $|w_{i}-w_{j}|\geq2^{m}r\min\left(\rho(w_{i}),\rho(w_{j})\right)$
whenever the two different points $w_{i}$ and $w_{j}$ are staying in the same subsequence. Let $\left\{a_{n}\right\}$ be such a subsequence and consider the measure $\nu$ defined by
$$d\nu=\left(\sum_{n}\chi_{n}\right) d\mu ,$$
where $\chi_{n}$ denotes the characteristic function of $D^{r}(a_{n})$. Since $k\geq2$, the discs $D^{r}(a_{n})$ are pairwise disjoints. Since $T_{\mu}\in\mathcal{S}_{p}(A_{\varphi}^{2})$ and $0 \leq \nu \leq \mu$, then
$0 \leq T_{\nu} \leq T_{\mu}$, which gives that $T_{\nu}\in\mathcal{S}_{p}(A_{\varphi}^{2})$ with $\left\|T_{\nu}\right\|_{\mathcal{S}_{p}(A_{\varphi}^{2})}
\leq\left\|T_{\mu}\right\|_{\mathcal{S}_{p}(A_{\varphi}^{2})}$.

Given an orthonormal basis
$\{e_{n}\}$ for $A_{\varphi}^{2}$, we define an operator $G$ on $A_{\varphi}^{2}$ by
\begin{align}\label{4444}
Gf=\sum_{n}\langle f,e_{n}\rangle_{A_{\varphi}^{2}} k_{a_{n}},\ \ f\in A_{\varphi}^{2}.
\end{align}
Then $G$ is bounded on $A_{\varphi}^{2}$ by Proposition \ref{atomicdecomposition}. Hence the operator $T=G^{\ast}T_{\nu}G$ is in $\mathcal{S}_{p}(A_{\varphi}^{2})$ with
\begin{align}\label{1212}
\|T\|_{\mathcal{S}_{p}(A_{\varphi}^{2})} \leq\|G\|^{2} \cdot\left\|T_{\nu}\right\|_{\mathcal{S}_{p}(A_{\varphi}^{2})}
\lesssim\left\|T_{\mu}\right\|_{\mathcal{S}_{p}(A_{\varphi}^{2})}.
\end{align}

By using  (\ref{4444}) and the relation
$$\langle Tf, g\rangle_{A_{\varphi}^{2}}=\langle T_{\nu}Gf, Gg\rangle_{A_{\varphi}^{2}},\quad f,g \in A_{\varphi}^{2},$$
we have
$$
T f=\sum_{n, j}\left\langle T_{\nu} k_{a_{n}}, k_{a_{j}}\right\rangle_{A_{\varphi}^{2}}\left\langle f, e_{n}\right\rangle_{A_{\varphi}^{2}} e_{j}, \quad f \in A_{\varphi}^{2},
$$

We decompose the operator $T$ as $T=T_{1}+T_{2}$, where $T_{1}$ is the diagonal operator defined by
$$
T_{1} f=\sum_{n}\left\langle T_{\nu} k_{a_{n}}, k_{a_{n}}\right\rangle_{A_{\varphi}^{2}}\left\langle f, e_{n}\right\rangle_{A_{\varphi}^{2}} e_{n}, \quad f \in A_{\varphi}^{2},
$$
and $T_{2}$ is the non-diagonal part defined by $T_{2}=T-T_{1}$. Apply  Rotfel'd inequality \cite{Thompson},
we obtain
\begin{align}\label{sp1}
\|T\|_{\mathcal{S}_{p}(A_{\varphi}^{2})}^{p} \geq\|T_{1}\|_{\mathcal{S}_{p}(A_{\varphi}^{2})}^{p}
-\|T_{2}\|_{\mathcal{S}_{p}(A_{\varphi}^{2})}^{p}.
\end{align}

Since $T_{1}$ is positive diagonal operator, by Lemma \ref{classicalequation}, estimates (\ref{ke12}), (\ref{ke21}) and \cite[Lemma 3.1]{Hu1}, we deduce
\begin{align}\label{sp10}
\|T_{1}\|_{\mathcal{S}_{p}(A_{\varphi}^{2})}^{p}\nonumber&=\sum_{n}\left\langle T_{\nu} k_{a_{n}}, k_{a_{n}}\right\rangle_{A_{\varphi}^{2}}^{p}
=\sum_{n}\left(\int_{\mathbb{D}}\left|k_{a_{n}}(z)\right|^{2} e^{-2\varphi(z)} d \nu(z)\right)^{p}\\
&\gtrsim\sum_{n}\left(\int_{D^{r}(a_{n})}\frac{1}{\rho(z)^{2}}d\mu(z)\right)^{p}
\gtrsim\sum_{n}\widehat{\mu}_{r}(a_{n})^{p}.
\end{align}

On the other hand, since $0 < p < 1$, by \cite[Proposition 1.29]{Zhu1}, Lemmas \ref{classicalequation} and \ref{latticelemma}, we have
\begin{align}\label{sp2}
\|T_{2}\|_{\mathcal{S}_{p}(A_{\varphi}^{2})}^{p} \nonumber&\leq \sum_{n} \sum_{k}\left\langle T_{2} e_{n}, e_{k}\right\rangle_{A_{\varphi}^{2}}^{p}=\sum_{n, k : k \neq n}\left\langle T_{\nu} k_{a_{n}}, k_{a_{k}}\right\rangle_{A_{\varphi}^{2}}^{p}\\
\nonumber&\leq\sum_{n, k : k \neq n}\left(\int_{\mathbb{D}}\left|k_{a_{n}}(\xi)\right|\left|k_{a_{k}}(\xi)\right| e^{-2\varphi(\xi)} d \nu(\xi)\right)^{p}\\
&\leq\sum_{n, k : k \neq n}\left(\sum_{j}\int_{D^{r}(a_{j})}\left|k_{a_{n}}(\xi)\right|\left|k_{a_{k}}(\xi)\right| e^{-2\varphi(\xi)} d \mu(\xi)\right)^{p}.
\end{align}

If $n \neq k$, then $|a_{n}-a_{k}|\geq2^{m}r\min\left(\rho(a_{n}),\rho(a_{k})\right)$. Therefore, for $\xi\in D^{r}(a_{j})$, it is easy to see that either
$$|a_{n}-\xi|\geq2^{m-2}r\min\left(\rho(a_{n}),\rho(\xi)\right)\ \ \text{or}\ \
|\xi-a_{k}|\geq2^{m-2}r\min\left(\rho(\xi),\rho(a_{k})\right).$$
Hence, for any $\xi\in D^{r}(a_{j})$, we may suppose that $|a_{n}-\xi|\geq2^{m-2}r\min\left(\rho(a_{n}),\rho(\xi)\right)$.

For any $n,k\in\mathbb{N}^{+}$, let
$$
J_{n k}(\mu)=\sum_{j} \int_{D^{r}\left(a_{j}\right)}
\left|k_{a_{n}}(\xi)\right|\left|k_{a_{k}}(\xi)\right| e^{-2\varphi(\xi)} d\mu(\xi).
$$
Putting this notation into estimate (\ref{sp2}), we obtain
\begin{align}\label{sp3}
\|T_{2}\|_{\mathcal{S}_{p}(A_{\varphi}^{2})}^{p} \leq \sum_{n, k : k \neq n} J_{n k}(\mu)^{p}.
\end{align}

Taking $N$ large enough in estimate (\ref{(Carlesonlemma2)}) and using the fact that $|a_{n}-\xi|\geq2^{m-2}r\min\left(\rho(a_{n}),\rho(\xi)\right)$, we have
\begin{align*}
\left|k_{a_{n}}(\xi)\right|e^{-\varphi(\xi)}
\lesssim\frac{1}{\rho(\xi)}\left(\frac{\min\left(\rho(a_{n}),\rho(\xi)\right)}
{\left|a_{n}-\xi\right|}\right)^{N}\lesssim\frac{1}{\rho(\xi)}2^{-Nm}.
\end{align*}
Thus, apply this inequality  to the power $1/2$, we get
\begin{align}\label{sp4}
\left|k_{a_{n}}(\xi)\right|=\left|k_{a_{n}}(\xi)\right|^{1/2}
\left|k_{a_{n}}(\xi)\right|^{1/2}
\lesssim2^{-Nm/2}\frac{e^{1/2\varphi(\xi)}}
{\rho(\xi)^{1/2}}\left|k_{a_{n}}(\xi)\right|^{1/2}.
\end{align}

On the other hand, by estimates (\ref{ke11}) and (\ref{ke21}),
we obtain
\begin{align}\label{sp5}
\left|k_{a_{k}}(\xi)\right|=\frac{\left|K(\xi,a_{k})\right|^{1/2}}
{\|K_{a_{k}}\|_{A_{\varphi}^{2}}^{1/2}}\left|k_{a_{n}}(\xi)\right|^{1/2}
\lesssim\frac{e^{1/2\varphi(\xi)}}
{\rho(\xi)^{1/2}}\left|k_{a_{k}}(\xi)\right|^{1/2}.
\end{align}
Then estimates (\ref{sp4}), (\ref{sp5}) and \cite[Lemma 3.1]{Hu1} yield that
$$
J_{n k}(\mu) \lesssim 2^{-Nm/2} \sum_{j} \frac{1}{\rho(a_{j})} \int_{D^{r} \left( a_{j}\right)}\left|k_{a_{n}}(\xi)\right|^{1 / 2}\left|k_{a_{k}}(\xi)\right|^{1 / 2} e^{-\varphi(\xi)} d \mu(\xi).
$$
In views of Lemma \ref{submeanlemma}, (\ref{vip}) and \cite[Lemma 3.1]{Hu1}, for $\xi\in D^{r}(a_{j})$, one has
\begin{align*}
\left|k_{a_{n}}(\xi)\right|^{1 / 2} e^{-1 / 2\varphi(\xi)}
\lesssim\left(\frac{1}{\rho(\xi)^{2}}\int_{D^{r}\left(\xi\right)}
\left|k_{a_{n}}(z)\right|^{p / 2} e^{-p / 2\varphi(z)}dA(z)\right)^{1/p}
\lesssim\rho(a_{j})^{-2/p}S_{n}(a_{j})^{1/p},
\end{align*}
where
$$S_{n}(\cdot)=\int_{D^{2\alpha}(\cdot)}
\left|k_{a_{n}}(z)\right|^{p / 2} e^{-p / 2\varphi(z)}dA(z).$$

By a similar discussion, we also have
\begin{align*}
\left|k_{a_{k}}(\xi)\right|^{1 / 2} e^{-1 / 2\varphi(\xi)}
\lesssim\rho(a_{j})^{-2/p}S_{k}(a_{j})^{1/p}.
\end{align*}
Therefore,  since $N$ is taken large enough, we obtain
\begin{align*}
J_{n k}(\mu)&\lesssim 2^{-Nm/2} \sum_{j} \frac{\rho(a_{j})^{-4/p}}{\rho(a_{j})} S_{n}(a_{j})^{1/p} S_{k}(a_{j})^{1/p}\mu\left(D^{r}(a_{j})\right)\\
&\leq2^{-m} \sum_{j}
\rho(a_{j})^{1-4/p}S_{n}(a_{j})^{1/p}S_{k}(a_{j})^{1/p}
\widehat{\mu}_{r}(a_{j}).
\end{align*}

Since $0 < p < 1$, using the fact (\ref{fact0}), we have
\begin{align*}
J_{n k}(\mu)^{p}\lesssim2^{-m p} \sum_{j} \rho(a_{j})^{p-4}  S_{n}(a_{j})  S_{k}(a_{j})  \widehat{\mu}_{r}(a_{j})^{p}.
\end{align*}
This together with estimate (\ref{sp3}) give
\begin{align}\label{sp6}
\left\|T_{2}\right\|_{\mathcal{S}_{p}(A_{\varphi}^{2})}^{p} \lesssim 2^{-m p} \sum_{j} \rho(a_{j})^{p-4} \widehat{\mu}_{r}(a_{j})^{p}\left(\sum_{n} S_{n}(a_{j})\right)\left(\sum_{k} S_{k}(a_{j})\right).
\end{align}

On the other hand, we have
\begin{align}\label{sp9}
\sum_{k}S_{k}(a_{j})
=\int_{D^{2\alpha}(a_{j})}\left(\sum_{k}\left|k_{a_{k}}(z)\right|^{p/2} \right)e^{-p/2\varphi(z)}dA(z).
\end{align}

Now, we assert that
\begin{align}\label{sp7}
\sum_{k}\left|k_{a_{k}}(z)\right|^{p/2}\lesssim e^{p/2\varphi(z)}\rho(z)^{-p/2}.
\end{align}

Applying estimate (\ref{(Carlesonlemma1)}), \cite[Lemma 3.1]{Hu1} and $(c)$ of Lemma \ref{latticelemma}, we obtain
\begin{align}\label{sp8}
\sum_{a_{k}\in D^{r_{0}}(z)}\left|k_{a_{k}}(z)\right|^{p/2}
\lesssim e^{p/2\varphi(z)}\sum_{a_{k}\in D^{r_{0}}(z)}\rho(a_{k})^{-p/2}
\lesssim e^{p/2\varphi(z)}\rho(z)^{-p/2}
\end{align}
for some $r_{0}>0$.

On the other hand, taking $N$ in estimate (\ref{(Carlesonlemma2)}) large enough such that $Np/2-2>0$, it follows that
\begin{align*}
\sum_{a_{k}\notin D^{r_{0}}(z)}\left|k_{a_{k}}(z)\right|^{p/2}
&\lesssim e^{p/2\varphi(z)}\rho(z)^{Np/2-p/2-2}\sum_{a_{k}\notin D^{r_{0}}(z)}\frac{\rho(a_{k})^{2}}{|z-a_{k}|^{Np/2}}\\
&=e^{p/2\varphi(z)}\rho(z)^{Np/2-p/2-2}\sum_{j=0}^{\infty} \sum_{a_{k} \in R_{j}(z)}\frac{\rho(a_{k})^{2}}{|z-a_{k}|^{Np/2}},
\end{align*}
where
$$
R_{j}(z)=\left\{\zeta \in \mathbb{D} : 2^{j} r_{0} \rho(z)\leq|\zeta-z| < 2^{j+1} r_{0} \rho(z)\right\},\ j=0,1,2 \ldots
$$

By the definition of $\mathcal{L}_{0}$, it's not difficult to see that, for
$j=0,1,2, \ldots$,
$$
D^{r_{0}}(a_{k}) \subset D^{5r_{0}2^{j}}(z) \quad \text { if } \quad a_{k} \in D^{r_{0}2^{j+1}}(z).
$$
Hence, this combines with Lemma \ref{latticelemma} gives
$$
\sum_{a_{k} \in R_{j}(z)} \rho\left(a_{k}\right)^{2}
\lesssim\left|D^{5r_{0}2^{j}}(z)\right|
\lesssim 2^{2 j} \rho(z)^{2}.
$$
Hence
\begin{align*}
\sum_{a_{k}\notin D^{r_{0}}(z)}\left|k_{a_{k}}(z)\right|^{p/2}
&\lesssim e^{p/2\varphi(z)}\rho(z)^{-p/2-2}\sum_{j=0}^{\infty}2^{-Npj/2}
\sum_{a_{k} \in R_{j}(z)}\rho(a_{k})^{2}\\
&\lesssim e^{p/2\varphi(z)}\rho(z)^{-p/2}\sum_{j=0}^{\infty}2^{\frac{(4-Np)}{2}j}\lesssim e^{p/2\varphi(z)}\rho(z)^{-p / 2}.
\end{align*}
This together with (\ref{sp8}) yields the assertion (\ref{sp7}).

Putting  estimate (\ref{sp7}) into (\ref{sp9}), we obtain
\begin{align}\label{5555}
\sum_{k} S_{k}(a_{j}) \lesssim \rho(a_{j})^{2-p / 2}.
\end{align}

In a same manner we get
\begin{align}\label{6666}
\sum_{n}S_{n}(a_{j})  \lesssim \rho(a_{j})^{2-p/ 2}.
\end{align}

Putting  estimates (\ref{5555}) and (\ref{6666}) into (\ref{sp6}), we finally obtain
$$\left\|T_{2}\right\|_{\mathcal{S}_{p}(A_{\varphi}^{2})}^{p} \lesssim 2^{-m p} \sum_{j} \widehat{\mu}_{r}(a_{j})^{p}
\leq \frac{1}{2} \sum_{j} \widehat{\mu}_{r}(a_{j})^{p}$$
when $m$ is chosen big enough. Putting this and estimate (\ref{sp10}) into (\ref{sp1}), we get
$$\sum_{j} \widehat{\mu}_{r}(a_{j})^{p}\lesssim
\|T\|_{\mathcal{S}_{p}(A_{\varphi}^{2})}^{p}.$$
Since this holds for each  of the $M$ subsequences of $\{w_{n}\}$, and by estimate (\ref{1212}), we have
\begin{align}\label{sp11}
\sum_{n} \widehat{\mu}_{r}(w_{n})^{p}\lesssim
M\|T\|_{\mathcal{S}_{p}(A_{\varphi}^{2})}^{p}\lesssim M\left\|T_{\mu}\right\|_{\mathcal{S}_{p}(A_{\varphi}^{2})}<\infty.
\end{align}
This completes the proof of  Theorem \ref{Schattenclass}.
\end{proof}

\end{document}